\documentclass[11pt,reqno]{amsart}

\theoremstyle{plain}
\newtheorem{theorem} {Theorem}[section]
\newtheorem{lemma}[theorem] {Lemma}
\newtheorem{proposition}[theorem] {Proposition}
\newtheorem{corollary}[theorem] {Corollary}
\newtheorem{property}[theorem] {Property}

\theoremstyle{definition}
\newtheorem{definition}[theorem] {Definition}

\newtheorem{example} [theorem]{Example}

\theoremstyle{remark}
\newtheorem{remark}[theorem] {Remark}

\numberwithin{equation}{section}

\usepackage{amsfonts}
\usepackage{amssymb}
\usepackage{amscd}
\usepackage{bigints}

\newcommand{\R}{{\mathbb R}}
\newcommand{\Z}{{\mathbb Z}}
\newcommand{\N}{{\mathbb N}}
\newcommand{\NN}{{\mathcal N}}
\newcommand{\QQ}{{\mathcal Q}}

\newcommand{\PP}{{\mathcal P}}

\newcommand{\LL}{{\mathcal L}}

\newcommand{\C}{{\mathcal C}}
\newcommand{\CC}{{\mathbb C}}

\newcommand{\TT}{{\mathcal T}}

\newcommand{\ES}{{\mathfrak S}}

\newcommand{\BB}{{\mathfrak B}}
\newcommand{\al}{{\alpha}}
\newcommand{\la}{{\lambda}}
\newcommand{\sa}{{\sigma}}

\newcommand{\iy}{{\infty}}
\newcommand{\vphi}{{\varphi}}
\newcommand{\vep}{{\varepsilon}}
\newcommand{\g}{{\gamma}}
\newcommand{\de}{{\delta}}
\newcommand{\Om}{{\Omega}}

\newcommand{\be}{{\beta}}

\newcommand{\La}{{\Lambda}}

\newcommand{\bna}{\begin{eqnarray}}
\newcommand{\ena}{\end{eqnarray}}
\newcommand{\ba}{\begin{eqnarray*}}
\newcommand{\ea}{\end{eqnarray*}}
\newcommand{\beq}{\begin{equation}}
\newcommand{\eeq}{\end{equation}}
\DeclareMathOperator*{\esssup}{ess\,sup}

\begin{document}

\title[Discretization Theorems]
{Discretization Theorems for Entire Functions of Exponential Type}
\author{Michael I. Ganzburg}
 \address{212 Woodburn Drive, Hampton,
 VA 23664\\USA}
 \email{michael.ganzburg@gmail.com}

 \begin{abstract}
 We prove
$L_q(\R^m)$--discretization inequalities
  for entire functions $f$ of exponential type
  in the form
  \ba
  C_2\|f\|_{L_q(\R^m)}
 \le \left(\sum_{\nu=1}^\iy \left\vert f\left(X_\nu\right)
 \right\vert^q\right)^{1/q}
 \le C_1\|f\|_{L_q(\R^m)},\qquad q\in[1,\iy],
 \ea
 with estimates for $C_1$ and $C_2$.
 We find a necessary and sufficient condition on
 $\Omega=\left\{X_\nu\right\}_{\nu=1}^\iy\subset\R^m$
 for the right inequality to be valid and a
 sufficient condition on $\Omega$ for the left one to
 hold true.
 In addition, $L_\iy(Q^m_b)$-discretization
 inequalities on an $m$-dimensional cube are proved for
 entire functions of exponential type
 and exponential polynomials.
 \end{abstract}
 \maketitle

 \section{Introduction}\label{S1}
\setcounter{equation}{0}
\noindent
In this paper we prove
discretization theorems
(often called Marcinkiewicz or Marcinkiewicz--Zygmund type
inequalities)
  for entire functions of exponential type (EFETs)
  on $\R^m$ and on an $m$-dimensional cube.
\subsection{Notation and Definitions}\label{S1.1n}
Let $\R^m$ be the Euclidean $m$-dimensional space with elements
$x=(x_1,\ldots,x_m),\, y=(y_1,\ldots,y_m),
\,t=(t_1,\ldots,t_m)$,
the inner product $(t,x):=\sum_{j=1}^mt_jx_j$,
and the norm $\vert x\vert:=\sqrt{(x,x)}$.
Next, $\CC^m:=\R^m+i\R^m$ is the $m$-dimensional complex
space with elements
$z=(z_1,\ldots, z_m)=x+iy,\,w=(w_1,\ldots, w_m)$,
the symmetric bilinear form $(z,w):=\sum_{j=1}^mz_jw_j$, and
the norm $\vert z\vert:=\sqrt{\vert x\vert^2+\vert y\vert^2}$.
In addition, $\N:=\{1,\,2,\ldots\};\,
\Z^m$ denotes the set of all integral lattice points in $\R^m$;
$\Z^m_+$ is a subset of $\Z^m$
of all points with nonnegative coordinates;
 $\mathring{S}$ is the interior of
 a set $S\subseteq\R^m$;
and the symbol $\mathrm{card}(G)$ represents
the cardinal number of a finite set $G$.
We also use multi-indices $k=(k_1,\ldots,k_m)\in \Z^m_+$
with
 \ba
 \langle k\rangle:=\sum_{j=1}^m k_j,\qquad
  x^k:=x_1^{k_1}\cdot\cdot\cdot x_m^{k_m},\qquad
 D^k:=\frac{\partial^{k_1}}{\partial x_1^{k_1}}\cdot\cdot\cdot
 \frac{\partial^{k_m}}{\partial x_m^{k_m}}.
 \ea
 We also use the following standard norms on $\R^m$:
\ba
  \|x\|_1:=\sum_{j=1}^m\left\vert x_j\right\vert,\qquad
   \|x\|_2:=\vert x\vert,\qquad
  \|x\|_\iy:=\max_{1\le j\le m}\left\vert x_j\right\vert,\qquad x\in\R^m.
 \ea
 Given $M>0$, these norms generate the following sets:
the $m$-dimensional octahedron
$O^m_M:=\left\{t\in\R^m: \|x\|_1\le M\right\}$,
the $m$-dimensional ball
$\BB^m_M:=\{t\in\R^m: \|x\|_2\le M\}$,
and the $m$-dimensional cube
$Q^m_M:=\left\{t\in \R^m: \|x\|_\iy\le M\right\}$
in $\R^m$, respectively.
We also need the following notations: $\BB^m_R(x):=x+\BB^m_R,\,
 Q_h^m(x):=x+Q^m_h,\,
 x\in\R^m$, and $[A,B]^m:=Q_h^m(x_0)$,
where $h:=(B-A)/2$ and $x_0:=((B+A)/2,\ldots,(B+A)/2)$.

 Next, let $V$
  be a centrally symmetric (with respect to the origin)
 closed
 convex body in $\R^m$
 with the \emph{width} $w(V)$, and
 the \emph{diameter} $d(V)$.
 Here, $w(V)$
 is the minimum distance between two parallel supporting
  hyperplanes of $V$
 and $d(V)$ is
 the maximum distance between two points of $V$.
 In addition, let
 $V^*:=\{y\in\R^m: \forall\, t\in V, \vert (t, y)\vert \le 1\}$
 be the \emph{polar} of $V$. Then $V^{**}=V$ (see, e.g., \cite[Sect. 14]{R1970})
 and the following relation is valid
 (see \cite[Eqn. (1.6)]{G2023}):
 $w(V^*)=4/d(V)$. In particular,
 \beq\label{E1.0}
 w\left(O^m_{1/M}\right)=w\left(\left(Q^m_{M}\right)^*\right)=2/\left(M\sqrt{m}\right),
 \qquad d\left(O^m_{1/M}\right)=d\left(\left(Q^m_{M}\right)^*\right)=2/M.
 \eeq

 In addition, $\vert S\vert_l$ denotes the $l$-dimensional
 Lebesgue measure
of a $l$-dimensional measurable set $S\subset\R^m,\,1\le l\le m$.
We also use the floor function
 $\lfloor a \rfloor$ and the ceiling function
 $\lceil a \rceil$.

 Furthermore, let
 $L_{q}(S)$ be a space of all measurable
  complex-valued
  functions $F$ defined on a measurable set
  $S\subseteq\R^m$
  with the finite norm
 \ba
 \|F\|_{L_{q}(S)}
 :=\left\{\begin{array}{ll}
 \left(\int_S\vert F(x)\vert^q dx\right)^{1/q}, & 1\le q<\iy,\\
 \esssup_{x\in S} \vert F(x)\vert, &q=\iy.
 \end{array}\right.
 \ea
 In addition, $C(K)$ is a space of all continuous complex-valued
  functions $F$ defined on
 a compact $K\subset\R^m$
  with the finite norm
 $\|F\|_{C(K)}:=\max_{x\in K} \vert F(x)\vert$,
 and $C_{\R}(K)$ is a subspace of all real-valued
  functions from $C(K)$.

 \begin{definition}\label{D1.1}
 We say that a countable set $\Omega=\left\{X_\nu\right\}_{\nu=1}^\iy\subset\R^m$
 is a $\de$-covering net for $\R^m$, where $\de>0$, if for every $x\in\R^m$
 there exists $X_\nu\in\Omega$ such that
 $\left\|x-X_\nu\right\|_\iy<\de$.
 \end{definition}

 \begin{definition}\label{D1.2}
 We say that a countable set $\Omega=\left\{X_\nu\right\}_{\nu=1}^\iy\subset\R^m$
 is a $\de_1$-packing net for $\R^m$, where $\de_1>0$, if
 $\inf_{\nu\ne\mu}\left\|X_\nu-X_\mu\right\|_\iy\ge\de_1$.
 \end{definition}

 \begin{definition}\label{D1.2a}
 We say that a countable set $\Omega=\left\{X_\nu\right\}_{\nu=1}^\iy\subset\R^m$
 is a $(\de_1,N)$-packing net for $\R^m$, where $\de_1>0$, if
 there exists $N\in\Z^1_+$ such that
 $\sup_{\nu\in\N}\mathrm{card}\left(\Omega\cap
 \mathring{Q}_{\de_1/2}^m\left (X_\nu\right)\right)\le N+1$.
 \end{definition}
 \noindent
 Certain properties of $\de$-covering
 and $\de_1$-packing nets for $\R^m$ are discussed in Lemma \ref{L2.1b}.
 In particular, it follows from Definition \ref{D1.2a} and Lemma \ref{L2.1b} (b) that
 the definitions of a $\de_1$-packing net and $(\de_1,0)$-packing one for $\R^m$
 are equivalent.

 Given a bounded set $A\subset\R^m$, the set of all trigonometric polynomials
 $$T(x)=\sum_{\eta\in A\cap \Z^m}c_\eta\exp[i(\eta,x)]$$
  with complex
 coefficients is denoted by $\TT(A)$.
 In the univariate case we
  use the notation $\TT_n:=\TT([-n,n]),\,n\in\Z^1_+$.
   In cases of $A=Q_\sa^m$ and $A=\BB_\sa^m$, more general sets of entire
functions are defined below.

  \begin{definition}\label{D1.3}
 We say that an entire function $f:\CC^m\to \CC^1$ has exponential or spherical type $\sa,\,\sa>0$,
 if for any $\vep>0$ there exists a constant $C_0(\vep,f)>0$ such that
 for all $z\in \CC^m,\,
 \vert f(z)\vert\le C_0(\vep,f)\exp\left(\sa(1+\vep)\sum_{j=1}^m\left\vert z_j\right\vert\right)$ or
 $\vert f(z)\vert\le C_0(\vep,f)\exp\left(\sa(1+\vep)\vert z\vert\right)$, respectively.
 \end{definition}
 \noindent
   The sets of all entire functions of exponential and spherical types $\sa$ are denoted
  by $B_{\sa,m}$ and $B_{\sa,m,S}$, respectively.
  In the univariate case we use the notation
  $B_\sa:=B_{\sa,1}=B_{\sa,1,S},\,\sa>0$. Note that $B_{\sa,m,S}\subseteq B_{\sa,m}$.

  Throughout the paper, if no confusion may occur,
  the same notation is applied to
  $f\in B_{\sa,m}$ or $f\in B_{\sa,m,S}$ and its restriction to $\R^m$ (e.g., in the form
  $f\in  B_{\sa,m}\cap L_{q}(\R^m))$.
  The classes $B_{\sa,m}$ and $B_{\sa,m,S}$ were
  defined by Bernstein \cite{B1948} and Nikolskii
   (see e.g., \cite[Sects. 3.1, 3.2.6]{N1969} or
  \cite[Definition 5.1]{DP2010}), respectively.
  Certain standard properties of functions from $B_{\sa,m}$ are presented
   in Lemma \ref{L2.7}.

   In this paper we discuss two major classes of
   multivariate polynomials. The first one is the set
   $\PP_{n,m}$
   of all polynomials
   $P(x)=\sum_{\langle k\rangle\le n}c_k x^k$
   in $m$ variables with
  complex coefficients
   of total degree at most $n,\,n\in\Z^1_+$.
   The second one is the set
   $\QQ_{n,m}$
   of all polynomials
   $P(x)=\sum_{k\in \Z^m_+\cap Q^m_n}c_k x^k$
   in $m$ variables with
  complex coefficients
   of degree at most $n,\,n\in\Z^1_+$,
   in each variable.
    Both classes coincide in the univariate case, and we
  use the notation $\PP_n:=\PP_{n,1}$.
  In addition, we use the Chebyshev polynomial
   of the first kind
   \beq\label{E1.1}
T_n(u):=(1/2)\left(\left(u+\sqrt{u^2-1}\right)^n
+\left(u-\sqrt{u^2-1}\right)^n\right),\,u\in\R^1.
\eeq

  Throughout the paper $C,\,C_1,\ldots,C_{8}$
  denote positive constants independent
of essential parameters.
 Occasionally we indicate dependence on certain parameters.
 The same symbols $C,\,C_1,\,C_2$, and $C_3$ do not
 necessarily denote the same constants in different occurrences,
 while $C_l,\,4\le l\le 8$,
 denote the same constants in different occurrences.

 A short survey, main results, and an outline of the proofs are
 presented in Sections \ref{S1.2n}--\ref{S1.4n}, respectively.

  \subsection{Discretization Theorems}\label{S1.2n}
  Let $B$ be a vector space of measurable functions on a
  measurable set $S\subseteq\R^m$
  and let $q\in[1,\iy]$.
  Discretization theorems
  for $B$ state that there exist a finite (for a bounded $S$)
  or countable (for an unbounded $S$) set of knots
  $\Omega=\left\{X_\nu\right\}_{\nu\in\Delta}\subset S$
  and constants
  $C_1=C_1(B,\Omega,S,q,m)
  \ge C_2=C_2(B,\Omega,S,q,m)$
  such that for all $f\in B$,
  \beq\label{E1.3.1b}
  C_2\|f\|_{L_q(S)}
 \le \left(\sum_{\nu\in\Delta} \left\vert f\left(X_\nu\right)
 \right\vert^q\right)^{1/q}
 \le C_1\|f\|_{L_q(S)}.
 \eeq
 For $q=\iy$, the right inequality of \eqref{E1.3.1b} is trivial
 with $C_1=1$, and
 the left one with $C_2\in(0,1)$ can be written in the form
 \beq\label{E1.3.1c}
 \|f\|_{L_\iy(S)}
 \le (1+\g)\sup_{\nu\in\Delta} \left\vert f\left(X_\nu\right)
 \right\vert,
 \eeq
where $\g>0$.
In case of a compact set $S,\,q\in[1,\iy)$,
and $\mathrm{card}(\Omega)=\La$, inequalities
\eqref{E1.3.1b} are often written in the following form
of so-called  Marcinkiewicz-type theorems
(see, e.g., \cite[Eqn. (1.2)]{DPTT2019}):
\beq\label{E1.3.1d}
  C_2\|f\|_{L_q(S)}
 \le \left(\frac{1}{\La}
 \sum_{\nu=1}^\La
  \left\vert f\left(X_\nu\right)
 \right\vert^q\right)^{1/q}
 \le C_1\|f\|_{L_q(S)},
 \eeq
where $\La=\La(B,\Omega,S,q,m)$ and
$C_1=C_1(q,m)
  \ge C_2=C_2(q,m)$.
  In certain cases, the set $\Omega$ can be explicitly
  identified.
  "Good" estimates of $\La$ are needed for each
  discretization theorem with a compact $S$.
  In this paper, we often call inequalities
   \eqref{E1.3.1b}, \eqref{E1.3.1c},
  and \eqref{E1.3.1d} discretization theorems.

Discretization theorems had been initiated
in the 1930s--1940s
 by Bernstein (1931 and 1948), Cartwright (1936),
Marcinkiewicz (1936), Marcinkiewicz
and Zygmund (1937),
Duffin and Shaeffer (1945), and others.
Influenced by problems of metric entropy,
numerical integration, and interpolation,
this topic has revisited in the  1990s--2020s
(see detailed surveys by Lubinsky \cite{L1998},
Schmeisser and Sickel \cite{SS2000},
Bos et al. \cite{BMSV2011}, Dai et al. \cite{DPTT2019},
Kro\'{o} \cite{K2021},
recent papers by Temlyakov  \cite{T2017},
Dai et al. \cite{DPSTT2021},  Kro\'{o} \cite{K2022},
 and the references therein).

In most publications, containing discretization theorems,
the space $B$ is one of the following spaces:
real- or complex-valued trigonometric or algebraic
polynomials, EFETs, and exponential polynomials.

 \subsubsection{Trigonometric Polynomials}\label{S1.2.1n}
 The story begins, like many others in
approximation theory, with Bernstein in 1931
who proved \cite[Eqns. (6),(22)]{B1931}
the following inequalities for $S=[0,2\pi],\,B=\TT_n,\,q=\iy,\,
N\in\N,\,N>n$, and $T_n\in\TT_n$:
\bna
 \left\|T_n\right\|_{L_\iy([0,2\pi])}
 &\le& \sqrt{\frac{2N+1}{2N-2n+1}}
 \max_{0\le\nu\le 2N} \left\vert
  T_n\left(\frac{2\nu\pi}{2N+1}\right)
 \right\vert
 \nonumber\\
 &=& (1+O(n/N))\max_{0\le\nu\le 2N} \left\vert
  T_n\left(\frac{2\nu\pi}{2N+1}\right)
 \right\vert,\label{E1.3.1e}\\
 \left\|T_n\right\|_{L_\iy([0,2\pi])}
 &\le& \left(\cos \frac{n\pi}{2N}\right)^{-1}
 \max_{0\le\nu\le 2N-1} \left\vert
  T_n\left(\frac{\nu\pi}{N}\right)
 \right\vert
 \nonumber\\
 &=& (1+O(n/N))\max_{0\le\nu\le 2N-1} \left\vert
  T_n\left(\frac{\nu\pi}{N}\right)
 \right\vert,\label{E1.3.1f}
 \ena
 as $n/N\to 0$, where inequality \eqref{E1.3.1f} is sharp for
 $n\vert N$. Note that \eqref{E1.3.1c}
 immediately follows from \eqref{E1.3.1e} and \eqref{E1.3.1f}
 for a large enough $N$
 and, in addition,
  $\La=2N$ for \eqref{E1.3.1f}.

 A version of \eqref{E1.3.1e} for polynomials
 on the unit circle was obtained by Sheil-Small
 \cite[Theorem 1]{SS2008}.
 Dubinin \cite[Theorem 1]{D2011} and Kalmykov
 \cite[Theorems 1 and 2]{K2012}
 extended \eqref{E1.3.1f} for polynomials
 on the unit circle and on a circular arc,
 respectively.

 Discretization theorem \eqref{E1.3.1d}
 for $S=[0,2\pi]$ and $B=\TT_n$   was proved by
 Marcinkiewicz \cite[Theorems 9 and 10]{M1936}
 (see also \cite[Theorem 7.5]{Z1959}) for $q\in(1,\iy)$
 and $\La=2n+1$
 and by Marcinkiewicz and Zygmund \cite[Theorem 7]{MZ1937}
 (see also \cite[Eqn. 7.29]{Z1959}) for $q=1$
 and $\La=\lceil 2(1+\vep)n\rceil$ with a fixed $\vep>0$
 and all $n\in\N$.
 These results were extended to
 $S=[0,2\pi]^m,\,q\in[1,\iy]$,
 and multivariate
 polynomials from $\TT(\Pi)$, where
 $\Pi:=\prod_{j=1}^m[-M_j,M_j],\,M_j\in\N,\,1\le j\le m$,
 is a $m$-dimensional parallelepiped.
 In particular,
 $\La=\mathrm{card}(\Pi\cap \Z^m)$ for $q\in(1,\iy)$;
 in case of $q=1$ and $q=\iy$, the estimate
 $\La\le C(m)\,\mathrm{card}(\Pi\cap \Z^m)$ is valid
 (see \cite[Sect. 2.1]{DPTT2019} for more details).

 The problem of proving discretization theorem
 \eqref{E1.3.1d} for $\TT(A)$ with "optimal" estimates for
 $\La$ is open in case of more intricate sets $A$. For instance,
 in case of the hyperbolic cross $A$, inequalities
 \eqref{E1.3.1d} with $\La\le C\mathrm{card}(A\cap\Z^m)$
 are known only for $q=2$ (see
 \cite[Theorem 1.1]{T2017}
 and \cite[Theorem 2.2]{DPTT2019}).

 \subsubsection{Algebraic Polynomials}\label{S1.2.2n}
  Discretization theorem \eqref{E1.3.1c}
 for $S=[-b,b],\,b>0$, and $B=\PP_n$
 immediately follows from \eqref{E1.3.1f}
 with $\La=2N$ by the standard substitution $x=b\cos t$
 (see also \cite[p. 91, Lemma 3 (iii)]{C1966}).

 Many discretization theorems have been established for
 multivariate polynomials
 (see, e.g., \cite{BMSV2011,K2011,K2019,K2021,DP2023} and
 references therein). In particular, the following
 general result was recently proved by Dai and Prymak
 \cite[Remark 2.4]{DP2023}:

 \begin{theorem}\label{T1.0}
 For any $\g >0$ there exists a constant $C=C(m,\g)$
 such that for
every $n\in\N$ and every convex body $\C\subset\R^m$,
there exists a set
$\left\{X_\nu\right\}_{\nu=1}^\La\subset\C$
with $\La \le Cn^m$ such that
\beq\label{E1.3.1g}
\|Q\|_{L_\iy(\C)}
 \le (1+\g)\max_{1\le\nu\le\La}
 \left\vert Q\left(X_\nu\right)
 \right\vert
\eeq
for every $Q\in\PP_{n,m}$.
 \end{theorem}
 Kro\'{o}  conjectured in \cite[p. 1118]{K2011}
  that the
  statement of Theorem \ref{T1.0} holds true for
  $1+\g$ replaced with $C$ in \eqref{E1.3.1g}.
  He proved the conjecture in \cite[Theorems 1--4]{K2011}
   for all convex polytopes
  and certain other domains in $\R^m$.
  In addition, he established
   Theorem \ref{T1.0} for $m=2$ in
   \cite[Main Theorem]{K2019}.

   \subsubsection{EFETs}\label{S1.2.3n}
   Unlike polynomials, the discretization of EFETs on $\R^m$
   obviously requires an infinite set of knots.

   Discretization theorems for EFETs in case of $q=\iy$ (i.e., inequalities
   in the form of \eqref{E1.3.1c}) have been known  since 1936.
   The celebrated result proved by Cartwright
   \cite{C1936} (see also \cite[Sect. 4.3.3]{T1963})
   states that for any $p>\sa>0$ and every
   $f\in B_\sa,\,
   \|f\|_{L_\iy(\R^1)}
 \le C(p,\sa)\sup_{\nu\in\Z^1} \left\vert f
 \left(\nu\pi/p\right)
 \right\vert$.
Duffin and Schaeffer \cite[Theorem 1]{DS1945}
 strengthened this result by proving that for every
 sequence $\Omega=\left\{X_\nu\right\}_{\nu=-\iy}^\iy
 \subset\R^1$,
 satisfying the conditions
 \beq\label{E1.3.1h}
 \sup_{\nu\in\Z^1}\left\vert X_\nu-\nu\pi/p\right\vert<\iy,
 \qquad
 \inf_{\nu\ne\mu}\left\vert X_\nu-X_\mu\right\vert>0,
 \eeq
 the following inequality holds:
 $\|f\|_{L_\iy(\R^1)}
 \le C(p,\sa,\Omega)\sup_{\nu\in\Z^1} \left\vert f
 \left(X_\nu\right)\right\vert$.
 The careful analysis of the proofs of these results shows that
 $\lim_{p\to\iy}C(p,\sa)=\lim_{p\to\iy}C(p,\sa,\Omega)=1$,
so  \eqref{E1.3.1c} is valid for $S=\R^1,\,B=B_\sa$,
 and $\Delta=\Z^1$.

 Bernstein \cite[Theorem 1]{B1948b} introduced a weakened condition
 (compared with \eqref{E1.3.1h})
 \beq\label{E1.3.1i}
 0<X_{\nu+1}-X_{\nu}\le \pi/p,\qquad\nu\in\Z^1,\quad p>\sa,
 \eeq
 that guarantees the validity of
 the following nonperiodic version of \eqref{E1.3.1f}:
 \ba
 \|f\|_{L_\iy(\R^1)}
 \le \left(\cos \frac{\sa\pi}{2p}\right)^{-1}
 \sup_{\nu\in\Z^1} \left\vert f
 \left(X_\nu\right)\right\vert
 \ea
  for every $f\in B_\sa$ of
 at most polynomial growth on $\R^1$.
 More univariate Cartwright-type theorems can be found in
 \cite{BU2016} and in the references therein.

Logvinenko \cite[Theorem 1]{L1990}
 proved a multivariate Cartwright-type theorem,
 replacing univariate conditions
   \eqref{E1.3.1h} and \eqref{E1.3.1i}
 with the condition that
 $\left\{X_\nu\right\}_{\nu=1}^\iy\subset\R^m$ is a
 $\de$-covering net for $\R^m$.
 His result states that if
 $\de\sa<(2(\lceil em\rceil+1)^{-1}$, then for every
 $f\in B_{\sa,m}$,
 \beq\label{E1.3.1j}
 \|f\|_{L_\iy(\R^m)}
 \le e^\de(1-\de\sa)^{-1}
 \sup_{\nu\in\N} \left\vert f
 \left(X_\nu\right)\right\vert.
 \eeq
 Thus \eqref{E1.3.1c} holds  for $S=\R^m,\,B=B_{\sa,m}$,
 and $\Delta=\N$.
 Earlier versions of this result
 with a stronger condition
  were obtained in Theorem 1 of
 \cite{L1974,L1975}.
 Note that a $\de$-covering net is defined differently in
 \cite{L1974,L1975,L1990} for $m>1$; namely,
 the norm $\|\cdot\|_\iy$ of Definition \ref{D1.1}
 is replaced by $\|\cdot\|_1$ in these publications.

 The celebrated univariate discretization theorem
for $f\in B_\sa\cap L_q(\R^1),\,q\in (1,\iy)$,
was proved by Plancherel and P\'{o}lya \cite{PP1937} in the following form:
\ba
  C_2\|f\|_{L_q(\R^1)}
 \le \left(\sum_{\nu\in\Z^1} \left\vert f\left(\pi \nu/\sa\right)
 \right\vert^q\right)^{1/q}
 \le C_1\|f\|_{L_q(\R^1)}.
 \ea
 Its extension to $q\in (0,\iy]$ was discussed in \cite[Sects. 2.3, 2.5]{SS2000}
 (see also multivariate versions in \cite[Sect. 3.3.2]{N1969} and \cite[Sect. 1.4.4]{ST1987}).

 Multivariate discretization theorems
 \beq\label{E1.3.1iii}
  C_2\|f\|_{L_q(\R^m)}
 \le \left(\sum_{\nu\in\Z^1} \left\vert f\left(X_\nu\right)
 \right\vert^q\right)^{1/q}
 \le C_1\|f\|_{L_q(\R^m)}
 \eeq
  were discussed
 by Pesenson \cite[Theorem 3.1]{P2007} and
 by Zhai et al. in recent publication \cite[Theorem 4.1]{ZWW2024}.

 In particular, Pesenson proved that there exists
 $\Omega=\left\{X_\nu\right\}_{\nu=1}^\iy\in\R^m$
 such that for all functions
$f\in B_{\sa,m}\cap L_q(\R^m),\,q\in[1,\iy]$,
  inequalities \eqref{E1.3.1iii} are held.
Note that Pesenson actually proved a more general result,
  replacing $f\left(X_\nu\right)$ in \eqref{E1.3.1iii} with special
  compactly supported distributions
  $\Phi_\nu(f)$.

  Zhai et al. proved  \eqref{E1.3.1iii}
  for $f\in B_{\sa,m,S}\cap L_q(\R^m),\,q\in (0,\iy)$,
  and special sets
  $\Omega=\left\{X_\nu\right\}_{\nu=1}^\iy$.
  Note that the authors actually proved a weighted version of \eqref{E1.3.1iii}.
  Certain necessary conditions on a set $\Omega$ are discussed in
  \cite[Lemma 4.3]{ZWW2024} as well.
  The detailed comparison of our and the authors' results
  are given in Remark \ref{R1.5a}.

  A general approach to discretization theorems in various Banach and quasi-Banach spaces was developed by Kolomoitsev and Tikhonov in recent preprint \cite{KT2024}. In particular, the authors obtained discretization theorems
  for EFETs from $L_q(\R^m)$ and other spaces (see \cite[Theorem 6.2 and Example 6.4 (i)]{KT2024}).

  Note that the major difference between
  discretization theorems from \cite{PP1937,P2007,ZWW2024,KT2024} described above
  and our result (see Theorem \ref{T1.1} below) is that we do not include  the condition
  $f \in L_q(\R^m)$. The absence of this condition, on the one hand, strengthens discretization theorem \eqref{E1.3.1iii} but, on the other hand, it makes the proof of the left-hand side of \eqref{E1.3.1iii} more complicated.

  While discretization theorems for general EFETs on $\R^m$ have been known since  1936,
  the corresponding results on compact subsets of $\R^m$ are  unknown. Certainly, they are known for some special classes, e.g., for trigonometric polynomials on parallelepipeds (see Sect. \ref{S1.2.1n}). One more special class of EFETs is discussed below.

 \subsubsection{Exponential Polynomials}\label{S1.2.4n}
 Kro\'{o} \cite[Theorems 6 and 7]{K2022}
 established discretization theorems in the form of
  \eqref{E1.3.1c} for every real-valued multivariate
   exponential polynomial
$E(w, \al)=\sum_{l=1}^\NN \boldsymbol{c}_l e^{(\la_l,w)}$
with the separation condition
$\left\vert \la_l-\la_j\right\vert
\ge\al>0,\,l\ne j$.
The discretization theorems are proved for
convex polytopes and convex polyhedral cones $S$ in $\R^m$
with
\beq\label{E1.3.1k}
\La=\mathrm{card}(\Omega)
\le C(S,m)\left(\frac{\NN}{\sqrt{\g}}\right)^m
\log^m\left(\frac{A_{\NN}}{\al\g}\right),\qquad
A_{\NN}:=\max_{1\le l\le \NN}
\left\vert \la_l\right\vert.
\eeq
Since the exponential type $A_\NN$ of $E(w, \al)$
appears only in the logarithmic term of inequality
\eqref{E1.3.1k}, the main part of the bound is
provided by $\NN^m$.

In addition, Kro\'{o} \cite[p. 72]{K2022} remarked
that if for every exponential polynomial
\beq\label{E1.3.1l}
E_N^*(w)=\sum_{\langle k\rangle\le N}c_k e^{(k,w)},
\qquad c_k\in\R^1,\quad
k\in \Z^m_+\cap O^m_N,
\eeq
of "total degree" $N$, the inequality
$\left\|E_N^*\right\|_{L_\iy(Q^m_b)}
\le (1+\g)
\max_{1\le j\le \La} \left\vert E_N^*(X_j\right)\vert$
holds,
where $S\subset\R^m$ is a compact set with
$\vert S\vert_m>0$ and
$\Omega=\left\{X_\nu\right\}_{\nu=1}^\Lambda\subset S$
is a discrete set, then
\beq\label{E1.3.1m}
\Lambda=\mathrm{card}(\Omega)
\ge C(S,m)\left(N/\sqrt{\g}\right)^m
\eeq
(see also \cite[Theorem 4]{K2022} for $m=1$).
Since for polynomials \eqref{E1.3.1l}
$\NN \sim N^m$ and $A_\NN\sim N$, estimate
\eqref{E1.3.1k} is sharp with respect
to $\NN$ up to the logarithmic term for $m=1$.
However, for $m>1$ the main part of the bound
\eqref{E1.3.1k} for polynomials \eqref{E1.3.1l}
is $N^{m^2}$ versus the lower estimate
$CN^m$ of \eqref{E1.3.1m}.

\subsubsection{}\label{S1.2.5n}
In this paper, we first prove discretization theorems
  in the forms of \eqref{E1.3.1b}
  and \eqref{E1.3.1c} for
 $S=\R^m,\,q\in[1,\iy]$, and $B=B_{\sa,m}$
 (see Theorem \ref{T1.1} and Corollary \ref{C1.2}).
 These inequalities extend and generalize
 the results discussed
 in Section \ref{S1.2.3n}.
 Preciseness of conditions for
 $\Omega
 :=\left\{X_\nu\right\}_{\nu=1}^\iy$ is discussed
 in Theorem \ref{T1.1a}.
 Next, we prove discretization theorem
 \eqref{E1.3.1c} for the cube $S=Q^m_b$
 with any $b>0$
 and $B=B_{\sa,m}$ with a "large" $\sa$
 (see Theorem \ref{T1.3}).
 Finally, we apply this result to prove
 a sharp discretization theorem
 in the spirit of Section \ref{S1.2.4n}
 for a more general class
 of exponential polynomials than \eqref{E1.3.1l}
 (see Theorem \ref{T1.5} and Remark \ref{R1.7}).
 In addition, note that Theorem \ref{T1.1} and
 Corollary \ref{C1.2} for $q=\iy$
 and Theorem \ref{T1.3}
 can be applied to trigonometric polynomials from
 Section \ref{S1.2.1n}
 to obtain certain discretization theorems
 (see Remark \ref{R1.10}).
 Note also that Theorem \ref{T1.0} for algebraic polynomials
 from Section \ref{S1.2.2n}
  is used to prove Theorem \ref{T1.3}.

\subsection{Main Results and Remarks}\label{S1.3n}
  We first discuss discretization results for EFETs on
 $\R^m$.

 \begin{theorem}\label{T1.1}
 Let $\de_1>0,\,\de>0,\,q\in[1,\iy]$, and
 \beq\label{E1.3.01}
 d:=\left\{\begin{array}{ll}
 1,&m=1,\\
 \lfloor m/q\rfloor+1,&m>1.
 \end{array}\right.
 \eeq
 In addition, let $\Omega
 :=\left\{X_\nu\right\}_{\nu=1}^\iy$
 be a countable set of points from $\R^m$
 and let $f\in B_{\sa,m}$.
 Then the following statements hold:\\
 (a) If $\Omega$ is a $(\de_1,N)$-packing net
 for $\R^m,\,N\in\Z^1_+$ (see Definition \ref{D1.2a}),
 and $q\in[1,\iy)$, then
 \beq\label{E1.3.1}
 \left(\sum_{\nu=1}^\iy \left\vert f\left(X_\nu\right)
 \right\vert^q\right)^{1/q}
 \le C_1\|f\|_{L_q(\R^m)},
 \eeq
 where
 \beq\label{E1.3.2}
 C_1=C_1(\de_1,\sa,m,q,N)
 \le \left(\de_1/2\right)^{-m/q}(N+1)^{1/q}
 \left(1+C(m,q)\max\left\{\de_1\sa,(\de_1\sa)^d\right\}\right).
 \eeq
 (b) Let $\Omega$ be a $\de$-covering net
 for $\R^m$ (see Definition \ref{D1.1}) with
 $\de$, satisfying the condition
 $\de\sa \le C(m,q)$.
 If
 $\left(\sum_{\nu=1}^\iy \left\vert f\left(X_\nu\right)
 \right\vert^q\right)^{1/q}
 <\iy, \,q\in[1,\iy)$, then
 \beq\label{E1.3.4}
 \left(\sum_{\nu=1}^\iy \left\vert f\left(X_\nu\right)
 \right\vert^q\right)^{1/q}
 \ge C_2\|f\|_{L_q(\R^m)},
 \eeq
 where
 \bna\label{E1.3.5}
 C_2&=&C_2(\de,\sa,m,q)\nonumber\\
 &\ge& (4\de)^{-m/q}2^{1/q-1}
 \left(1-C(m,q)\max\left\{(\de\sa)^q,(\de\sa)^{dq}\right\}\right)^{1/q}>0.
 \ena
 (c) Let $\Omega$ be a $\de$-covering net
 for $\R^m$  with
 $\de$, satisfying the condition
 $11m^{3/2}\de\sa\le 1$.
 If \linebreak
 $\sup_{\nu\in\N}\left\vert f\left(X_\nu\right)
 \right\vert<\iy$, then
 \beq\label{E1.3.5b}
 \sup_{\nu\in\N}\left\vert f\left(X_\nu\right)
 \right\vert
 \ge C_3\|f\|_{L_\iy(\R^m)},
 \eeq
 where
 $C_3=C_3(\de,\sa,m) \ge 1-m\de\sa$.
 \end{theorem}

 A simplified version of Theorem \ref{T1.1}
  presented below immediately follows from
  Theorem \ref{T1.1}.

 \begin{corollary}\label{C1.2}
 Let $\de_1>0,\,\de>0$, and $f\in B_{\sa,m}$.
 Then the following statements hold true:\\
 (a) If $\left\{X_\nu\right\}_{\nu=1}^\iy$
 is a $\de_1$-packing
 (see Definition \ref{D1.2}) and
  $\de$-covering net for $\R^m$
 with
 $\de$, satisfying the condition $\de\sa \le C(m,q)$,
 then for $q\in[1,\iy)$,
 \beq\label{E1.3.5d}
 C_2\|f\|_{L_q(\R^m)}
 \le \left(\sum_{\nu=1}^\iy \left\vert
 f\left(X_\nu\right)
 \right\vert^q\right)^{1/q}
 \le C_1\|f\|_{L_q(\R^m)}
 \eeq
 with estimates \eqref{E1.3.2} and \eqref{E1.3.5}
 for $C_1=C_1(\de_1,\sa,m,q,0)$
 and $C_2=C_2(\de,\sa,m,q)$,
 respectively.\\
 (b) For any $\g\in(0,1)$, there exists
 $\de=\de(\g,\sa,m)$ such that
 if $\left\{X_\nu\right\}_{\nu=1}^\iy$
 is a $\de$-covering net for $\R^m$, then
 \ba
 \|f\|_{L_\iy(\R^m)}
 \le (1+\g)\sup_{\nu\in\N}
 \left\vert f\left(X_\nu\right)\right\vert.
 \ea
 \end{corollary}
 The next corollary reduces conditions on $\Omega$
 compared with Corollary \ref{C1.2} (a).
 \begin{corollary}\label{C1.2a}
 Let $f\in B_{\sa,m},\,q\in[1,\iy)$,
 and $\Omega=\left\{X_\nu\right\}_{\nu=1}^\iy$
 be a $\de$-covering net for $\R^m$ with
 $\de$, satisfying the condition $\de\sa \le C(m,q)$.
 Then there exists a subset
 $\left\{Z_\mu\right\}_{\mu=1}^\iy$
 of $\Omega$ such that \eqref{E1.3.5d} holds
 with  $Z_\mu$ replacing $X_\mu,\,\mu\in\N$,
 and with estimates \eqref{E1.3.2} and \eqref{E1.3.5}
 for $C_1=C_1(\de,\sa,m,q,0)$
 and $C_2=C_2(2\de,\sa,m,q)$,
 respectively.
 \end{corollary}
 In the following theorem, we present certain
 necessary conditions
 on a set $\Omega$  for inequalities
 \eqref{E1.3.1}, \eqref{E1.3.4}, and \eqref{E1.3.5b}
 to be valid.

 \begin{theorem}\label{T1.1a}
  If $\Omega
 :=\left\{X_\nu\right\}_{\nu=1}^\iy$
 is a countable set of points from $\R^m$,
 then the following statements hold:\\
  (a) Let there exist a nontrivial function
  $f_0\in B_{\sa,m}\cap L_q(\R^m),
  \,q\in[1,\iy)$, such that
 inequality \eqref{E1.3.1} is valid
 for a certain set $\Omega$, a fixed number
 $C_1>0$,
 and any shifted function $f(\cdot):=
 f_0(\cdot-\al),\,\al\in\R^m$.
 Then for any $\de_1\in (0,1/(m\sa)]$ and
 \beq\label{E1.3.5c}
 0\le N
  \le \left\lfloor
  \left((1/2)C_1\left\|
  f_0\right\|_{L_\iy(\R^m)}/\left\|f_0
  \right\|_{L_q(\R^m)}\right)^{-q}
  \right\rfloor-1,
 \eeq
 the set $\Omega$ is a $(\de_1,N)$-packing net
 for $\R^m$.\\
 (b) Let $\Omega$ be a $(\de_1,N)$-packing net
 for $\R^m,\,N\in\Z^1_+$, and let
 inequality \eqref{E1.3.4} be valid
 for a fixed number $C_2>0$ and any function $f\in B_{\sa,m}$ with
 $\left(\sum_{\nu=1}^\iy \left\vert f\left(X_\nu\right)
 \right\vert^q\right)^{1/q}
 <\iy, q\in[1,\iy)$.
 Then there exists $C(m,q)$ such that
 for any $\de\in (\de^*,\iy)$,
 where
 \beq\label{E1.3.5e}
 \de^*:=\max\left\{\de_1,\frac{C(m,q)}{\sa}\left(\frac{N+1}{\de_1^m C_2^q}\right)^{1/(\g q-m)}\right\},
 \qquad \g:=\lfloor m/q\rfloor+1,
 \eeq
 $\Omega$ is a $\de$-covering net
 for $\R^m$.\\
 (c) Let
 inequality \eqref{E1.3.5b} be valid
 for a fixed number $C_3>0$ and any function $f\in B_{\sa,m}$ with
 $\sup_{\nu\in\N}\left\vert f\left(X_\nu\right)
 \right\vert<\iy$.
 Then  for any $\de\in \left(\left(C_3\sa\right)^{-1},\iy\right),\,
 \Omega$ is a $\de$-covering net
 for $\R^m$.
 \end{theorem}
 Next, we discuss discretization results for
 entire functions of high exponential type
 on the cube $Q^m_b$. Let us define the function
 \beq\label{E1.3.5a}
\psi(\tau):=\frac{\sqrt{1+\tau^2}}{\tau}
-\log\left(\tau+\sqrt{1+\tau^2}\right),
\qquad \tau\in(0,\iy),
\eeq
with the unique positive zero
$\g_0=1.5088\ldots$; note that $\psi(\tau)<0$ for $\tau> \g_0$ (see also Sect. \ref{S3.1.2n}).

\begin{theorem}\label{T1.3}
 Let $\ES:=\{n(N)\}_{N=1}^\iy$ be
 an increasing sequence of
 positive integers, and
 let fixed numbers $b>0$ and $\tau>\g_0$
  be independent of a given $n=n(N)\in \ES$.
In addition, let $f(\cdot)=f(N,\cdot)$
 be an entire function,
satisfying the inequality
\beq\label{E1.3.6}
\vert f(w)\vert\le D_n
\|f\|_{L_\iy(Q^m_b)}
\exp\left(\sa\sum_{j=1}^m
\left\vert w_j\right\vert\right),\qquad n\in\ES,
\quad w\in\CC^m,
\eeq
where $\sa={n}/{(mb\tau)}$ and
$D_n=D_n(b,\tau,m)$ are constants.
If
\beq\label{E1.3.7}
\lim_{N\to\iy}
D_{n(N)}e^{n(N)\psi(\tau)}=0,
\eeq
then for any $\g>0$, there exist a constant
$C=C(b,\tau,m,\g)$,
an integer $n_0=n_0(b,\tau,m,\g)\in\N$,
and a finite set  $\left\{X_1,\ldots, X_\La\right\}
\subset Q^m_b$
with $\La\le Cn^m$ such that
\beq\label{E1.3.8}
\|f\|_{L_\iy(Q^m_b)}
\le (1+\g)
\max_{1\le j\le \La} \left\vert f(X_j\right)\vert.
\eeq
for every entire $f$, satisfying \eqref{E1.3.6} for $n\ge n_0$
and \eqref{E1.3.7}.
\end{theorem}
It is easy to verify conditions \eqref{E1.3.6} and
\eqref{E1.3.7} for the following class of EFETs.

\begin{example}\label{Ex1.4}
Let $F_0$ be an entire function of one variable,
 satisfying the inequality
$\left\vert F_0(\xi)\right\vert\le Ce^{\vert\xi\vert},\,
\xi\in\CC^1,$ and, in addition,
$F_0\ge 0$ on $\R^m\setminus\{0\}$ and $F_0(0)>0$.
Then any function of the form
\beq\label{E1.3.9}
f(w):=\int_{Q^m_\sa}F_0((w,y))\,d\mu(y), \qquad w\in\CC^m,
\eeq
where $\sa>0$ and $\mu$ is a positive measure on $Q^m_\sa$,
is entire, and, moreover,
\ba
\vert f(w)\vert
\le \int_{Q^m_\sa}d\mu(y)
\max_{y\in Q^m_\sa}\left\vert F_0((w,y))\right\vert
\le C(f(0)/F_0(0))
\exp\left(\sa\sum_{j=1}^m
\left\vert w_j\right\vert\right),\qquad w\in\CC^m.
\ea
Hence for any function \eqref{E1.3.9},
inequality \eqref{E1.3.6} holds  with
$D_n=C/F_0(0),\,n\in\N$, and obviously \eqref{E1.3.7}
is valid as well. Note that here, $\sa>0$ is any number,
i.e., $\sa$ is not necessarily equal to $n/(mb\tau)$
as assumed in Theorem \ref{T1.3}.
\end{example}
In particular,
the set of all exponential polynomials
$\sum_{l=1}^\NN \boldsymbol{c}_l e^{(\la_l,w)}$
with nonnegative
coefficients is a subset of the family of all functions
\eqref{E1.3.9} with $F_0(\xi)=e^\xi,\,\xi\in\CC^1$, and
all positive discrete measures
$\mu$ on $Q^m_\sa$.
Therefore, Theorem \ref{T1.3} holds  for this set of
exponential polynomials by Example \ref{Ex1.4}.
For exponential polynomials with real or
complex coefficients, the problem of verifying
conditions \eqref{E1.3.6} and \eqref{E1.3.7}
(i.e., a discretization theorem
for these polynomials in view of
Theorem \ref{T1.3}) is more complicated.
A discretization theorem for certain
exponential polynomials with
complex coefficients is presented below.

\begin{theorem}\label{T1.5}
Let a fixed number $b>0$ be independent of a given $N\in\N$.
In addition, let
\beq\label{E1.3.10}
E_N(w):=\sum_{k\in \Z^m_+\cap Q^m_N}c_k e^{(k,w)},
\qquad c_k\in\CC^1,\quad
k\in \Z^m_+\cap Q^m_N,
\eeq
be an exponential polynomial of "degree" at most $N$ in each
variable.
Then for any $\g>0$, there exist a constant
$C=C(b,m,\g)$,
an integer $N_0=N_0(b,m,\g)\in\N$,
and a finite set  $\left\{X_1,\ldots, X_\La\right\}
\subset Q^m_b$
with $\La\le CN^m$ such that for $N\ge N_0$,
\ba
\left\|E_N\right\|_{L_\iy(Q^m_b)}
\le (1+\g)
\max_{1\le j\le \La} \left\vert E_N(X_j\right)\vert
\ea
for every exponential polynomial \eqref{E1.3.10}.
\end{theorem}

\begin{remark}\label{R1.5a}
After we had proved Theorems \ref{T1.1} and \ref{T1.1a},
we found very interesting paper \cite{ZWW2024} whose results
are related to these theorems. These results
are discussed in Section \ref{S1.2.3n}. In particular, we mention there that
\cite[Theorem 4.1]{ZWW2024} contains the unnecessary condition
$f \in L_q(\R^m),\,q\in[1,\iy)$, that is absent in Theorem \ref{T1.1}.
The proof of Theorem \ref{T1.1} (b) without this condition is much more difficult and requires special techniques developed here (see Section \ref{S1.4n}).
In addition, the discretization theorem for $q=\iy$ was not proved in \cite{ZWW2024}, while this case is discussed in Theorem \ref{T1.1} (c) and Theorem \ref{T1.1a} (c).
We also note that Theorem \ref{T1.1a} (a) presents a stronger version of the necessary condition on $\Omega$
for inequality
 \eqref{E1.3.1} to hold compared with \cite[Lemma 4.3]{ZWW2024}.
\end{remark}

\begin{remark}\label{R1.6}
Statements (a) of Theorems \ref{T1.1} and \ref{T1.1a}
show that the condition that
$\Omega$ is a $(\de_1,N)$-packing net
 for $\R^m$ is sufficient and necessary for inequality
 \eqref{E1.3.1} to hold with any
 $f\in B_{\sa,m}\cap L_q(\R^m),\,q\in[1,\iy)$,
  and a certain $C_1>0$.
 We do not know the corresponding criterion for
 \eqref{E1.3.4}. It appears plausible that
 the condition that
$\Omega$ is a $\de$-covering net
 for $\R^m$ is necessary for inequality
 \eqref{E1.3.4} to hold  with any
 $f\in B_{\sa,m},\,
 \left(\sum_{\nu=1}^\iy \left\vert f\left(X_\nu\right)
 \right\vert^q\right)^{1/q}
 <\iy, q\in[1,\iy)$, and a certain $C_2>0$.
 This condition is sufficient by Theorem \ref{T1.1} (b)
 and (c) for $q\in[1,\iy]$
 and necessary for $q=\iy$ by Theorem \ref{T1.1a} (c).
 In case of $m=1$ and $q=\iy$,  the precise criterion
 was found by Beurling \cite{B1989} (see also Blank--
 Ulanovskii \cite{BU2016}).
 A "weak" necessary condition for $m\ge 1$ and $q\in[1,\iy)$
  is discussed in
 Theorem \ref{T1.1a} (b).
\end{remark}

\begin{remark}\label{R1.8}
It is difficult to compare conditions and
 constants of Theorem \ref{T1.1} (c) and discretization
  theorem \eqref{E1.3.1j} because a $\de$-covering net
   is defined differently in these results (see Section
   \ref{S1.2.3n}).
\end{remark}

\begin{remark}\label{R1.9}
Note that Logvinenko \cite[Theorem 2]{L1974} announced
a version of Corollary \ref{C1.2} (a) with no estimates
for $C_1$ and $C_2$,
  but no proof has been provided since then.
 \end{remark}

\begin{remark}\label{R1.10}
Since $\TT(V)\subset B_{d(V)/2,m}$,
Theorems \ref{T1.1} (c) and \ref{T1.3} and
 Corollaries \ref{C1.2} (b) and \ref{C1.2a}
are valid for trigonometric polynomials from $\TT(V)$.
In addition, note that Theorem \ref{T1.1} (c) generalizes
 and/or strengthens
 the results from \cite{C1936,DS1945,L1974,L1975}
  (see Section
   \ref{S1.2.3n}).
\end{remark}

\begin{remark}\label{R1.7}
Since every exponential polynomial \eqref{E1.3.1l}
is a polynomial \eqref{E1.3.10}, the estimate
$\Lambda\le CN^m$ of Theorem \ref{T1.5}
 is sharp with respect to $N$ due to estimate
 \eqref{E1.3.1m} (see Section \ref{S1.2.4n}).
\end{remark}

\subsection{Outline of the Proofs}\label{S1.4n}
The proofs of Theorems \ref{T1.1}, \ref{T1.1a}, \ref{T1.3},
and \ref{T1.5} and Corollary \ref{C1.2a} are presented in Section \ref{S4n}.
 They are based, firstly, on certain geometric results and
 properties of algebraic polynomials
  and EFETs given in Section \ref{S2n} and, secondly,
  on approximation of EFETs by polynomials and
  other EFETs discussed in Section \ref{S3n}.

  Major ingredients of the proofs of many
  discretization theorems are  Markov-Bernstein
   type inequalities for polynomials and/or EFETs.
   We need them as well, and three
   Bernstein-type inequalities
    are presented in Lemma \ref{L2.7}.
    In particular, the proof of Theorem \ref{T1.1}
    is based on a technical discretization estimate
    of Lemma \ref{L2.8} that uses a
    Bernstein-type inequality for EFETs
    from Lemma \ref{L2.7} (b).

    While statement (a) of Theorem \ref {T1.1}
    follows from Lemmas \ref{L2.2} (a) and \ref{L2.8}
     (the former one is a certain result from
     combinatorial geometry of independent interest),
    the proof of statement (b) is more complicated
      because the premises of the statement
     do not include
     the assumption of $f\in L_q(\R^m)$
     that prevents us from using
     Bernstein-type inequalities and
     Lemma \ref{L2.8}. That is why we approximate $f$
     by special EFETs $f_n\in L_q(\R^m),\,n\in\N$, and
     study their properties (see Sect. \ref{S3.2n}). Then we prove
     statement (b) of Theorem \ref {T1.1} for functions $f_n$ by using
     geometric Lemmas \ref{L2.1b} and \ref{L2.2} and
     the discretization estimate of Lemma \ref{L2.8}.
     Passing to the limit as $n\to\iy$  completes the
     proof of this statement.
     Statement (c) of Theorem \ref {T1.1} is proved similarly
     but with fewer technicalities.

     Note that the idea of using similar functions belongs to
     Logvinenko  \cite{L1974,L1975,L1990} who applied them
     to the proof of a version of Theorem \ref {T1.1} (c)
     (see \eqref{E1.3.1j} and Remark \ref{R1.8}).
     However, we use different  techniques for
     constructing and studding those functions that allow us
     to extend the discretization theorem
     \eqref{E1.3.1j} to $q\in[1,\iy)$.

     In addition, note that the construction and properties of
     $f_n,\,n\in\N$, are based on an approximation estimate
     of a function $f\in B_{\sa,m}$ by algebraic polynomials
     from $\PP_{n,m}$
     (see Lemma \ref{L3.6}).
     The technique of approximation of
     univariate and multivariate EFETs by algebraic
     polynomials is developed in Section \ref{S3.1n}.
     These estimates extend and/or strengthen earlier
     results by Bernstein \cite{B1946,B1947},
      Logvinenko \cite{L1974,L1975},
      and the author \cite{G1982,G1991,G2019b}.

      The proof of Theorem \ref {T1.1a}
     is based on properties of $(\de_1,N)$-packing
     and $\de$-covering nets for $\R^m$.

To prove Theorem \ref {T1.3},
we use Lemma \ref{L3.7}
that discusses
approximation of a function $f\in B_{\sa,m}$
by algebraic polynomials from $\QQ_{n,m}$.
This step reduces discretization inequality
\eqref{E1.3.8} of Theorem \ref {T1.3}
to discretization inequality
\eqref{E1.3.1g} of Theorem \ref {T1.0}.

Theorem \ref {T1.5} immediately follows from
Theorem \ref {T1.3}
if exponential polynomials \eqref{E1.3.10}
with coefficients $c_k, \,k\in\Z^m_+\cap Q^m_N$,
satisfy special conditions \eqref{E1.3.6} and \eqref{E1.3.7}.
To prove these conditions, it suffices to estimate
$\sum_{k\in\Z^m_+\cap Q^m_N}
\left\vert c_k\right\vert$.
This step is accomplished
  by using two new inequalities for multivariate
  polynomials presented in Section \ref{S2n}
  (see Lemmas \ref{L2.4} (b)
   and \ref{L2.6}).

   In particular, Lemma \ref{L2.4} (b) of
independent interest discusses a sharp estimate
for partial derivatives of a polynomial from
$\QQ_{n,m}$ outside of a cube. This result is a
multivariate generalization of the celebrated
Chebyshev's inequality.
Note that Lemma \ref{L2.4} (b) is proved for
polynomials with complex coefficients
due to the use of general Lemma \ref{L2.2a}
of independent interest that
reduces numerous Markov-Bernstein-Nikolskii
type inequalities for
complex-valued functions to real-valued ones.

\section{Properties of Sets, Polynomials, and Entire Functions}
\label{S2n}
\setcounter{equation}{0}
\noindent
In this section we discuss certain geometric results
and some properties of
polynomials and entire functions of exponential type.
\subsection{Geometric Lemmas}\label{S2.1n}
 First, we discuss properties of $\de$-covering
 and $\de_1$-packing nets for $\R^m$
 (see Definitions \ref{D1.1} and \ref{D1.2}).

\begin{lemma}\label{L2.1b}
(a)
 $\Omega=\left\{X_\nu\right\}_{\nu=1}^\iy$
 is a $\de$-covering net for $\R^m$ if and only if
 $\bigcup_{\nu=1}^\iy \mathring{Q}_\de^m\left(X_\nu\right)=\R^m$
 for the family of open cubes
 $\left\{\mathring{Q}_\de^m\left(X_\nu\right)\right\}_{\nu=1}^\iy$.\\
 (b)
 $\Omega=\left\{X_\nu\right\}_{\nu=1}^\iy$
 is a $\de_1$-packing net  for $\R^m$ if and only if
  the family of open cubes
 $\left\{\mathring{Q}_{\de_1/2}^m\left(X_\nu\right)\right\}_{\nu=1}^\iy$
 is pairwise disjoint.\\
 (c) If $\Omega=\left\{X_\nu\right\}_{\nu=1}^\iy$
 is a $\de$-covering  net for $\R^m$,
 then there exists a subset
 $\Omega^*=\left\{Z_\mu\right\}_{\mu=1}^\iy$
 of $\Omega$ such that $\Omega^*$ is
 a $\de$-packing and a $2\de$-covering  net for $\R^m$.
 \end{lemma}
 \begin{proof}
 The proofs of statements (a) and (b)
 are simple and left as an exercise to
the reader.

To prove statement (c), we first need certain
inductive constructions.
 Let us set $Z_1:=X_1$ and assume that we can construct
 a set
 $\Omega^*_n:=\left\{Z_1,\ldots,Z_n\right\}\subset \Omega$
 such that the following relations hold:
 \bna
 && \inf_{1\le \nu,\mu\le n;\nu\ne\mu}
 \left\|Z_\nu-Z_\mu\right\|_
 {\iy}\ge\de,\label{E2.1b}\\
 &&\left\{X_1,\ldots,X_n\right\}\subset
 \bigcup_{\mu=1}^n\mathring{Q}_{\de}^m\left(Z_\mu\right).
 \label{E2.1c}
 \ena
 Note that relations \eqref{E2.1b} and \eqref{E2.1c} are
  trivially valid for $n=1$.
 The set $\Omega_n:=\Omega\setminus
 \bigcup_{\mu=1}^n\mathring{Q}_{\de}^m\left(Z_\mu\right)
 \ne \emptyset$
 is a subsequence of $\Omega$,
 and we choose $Z_{n+1}$ as the element of $\Omega_n$ with
 the smallest index.
 Then taking account of \eqref{E2.1b}, we see that
 \eqref{E2.1b} holds for
 $n$ replaced with $n+1$
 since $Z_{n+1}\in\Omega_n$.
In addition, if $X_{n+1}\notin
 \bigcup_{\mu=1}^n\mathring{Q}_{\de}^m\left(Z_\mu\right)$,
 then $X_{n+1}=Z_{n+1}
 \in\mathring{Q}_{\de}^m\left(Z_{n+1}\right)$
 by \eqref{E2.1c} and by the choice of $Z_{n+1}$.
 Therefore, \eqref{E2.1c} holds for
 $n$ replaced with $n+1$.
 Thus by induction, for all $n\in\N$, elements of
 $\Omega^*_n$ satisfy relations \eqref{E2.1b} and
 \eqref{E2.1c}.

 Next, setting $\Omega^*:=\bigcup_{n=1}^\iy\Omega^*_n$,
 we see by \eqref{E2.1b} that $\Omega^*$ is a
 $\de$-packing net for $\R^m$.
 Finally, since
 $\Omega \subset
 \bigcup_{\mu=1}^\iy\mathring{Q}_{\de}^m\left(Z_\mu\right)$
 by \eqref{E2.1c},
 for each $x\in\R^m$
 there exist elements $X_\nu\in\Omega$
 and $Z_\mu\in\Omega^*$, such that
 $\left\|x-X_\nu\right\|_{\iy}<\de$ and
 $\left\|X_\nu-Z_\mu\right\|_{\iy}<\de$.
 Therefore, $\Omega^*$ is
 a $2\de$-covering  net for $\R^m$.
 \end{proof}

Next, we discuss a certain result from
 combinatorial geometry.
 \begin{lemma}\label{L2.2}
 (a) Let $G:=\left\{S_1,S_2,\ldots\right\}$
 be a finite or countable family
 of sets from $\R^m$.
 Assume that there exists a number $N=N(G)\in\Z^1_+$
 such that any set from $G$ has
 the nonempty intersection with no more than $N$
 sets from $G$, not counting the set itself.
 Then
 there exists a partition
 $\left\{G_j\right\}_{j=1}^{N+1}$ of $G$
 with pairwise disjoint sets in each subfamily
 $G_j,\,1\le j\le N+1$
 (certain subfamilies can be empty).\\
 (b) Let $\Omega=\left\{X_\nu\right\}_{\nu=1}^\iy$ be a
 $\de_1$-packing net for $\R^m$.
 Then any cube from a family
 $G:=\left\{Q_\de^m\left(X_\nu\right)\right\}_{\nu=1}^\iy$
 of closed cubes with $\de_1<2\de$ has
 the nonempty intersection with no more than
 \beq \label{E2.2}
 N=\left\lfloor 2^m\left(\left({4\de}/
 {\de_1}\right)^m-1\right)\right\rfloor-1
 \eeq
 cubes from $G$, not counting the cube itself.\\
 (c) In addition to statement (b),
 there exists a partition $\left\{G_j\right\}_{j=1}^{N+1}$ of $G$
 with pairwise disjoint cubes in each subfamily $G_j,\,1\le j\le N+1$
 (certain subfamilies can be empty), where $N$ is defined by \eqref{E2.2}.
 \end{lemma}
 \begin{proof}
 (a) We first note that if $\mathrm{card}(G)<N+1$,
 then the statement trivially holds true
  with certain empty subfamilies.
  So assume that $\mathrm{card}(G)\ge N+1$.

  We start an inductive process of constructing a partition by setting
  $G_j(1):=\left\{S_j\right\},\,1\le j\le N+1$. Obviously,
  $\left\{G_j(1)\right\}_{j=1}^{N+1}$ is a partition of
  $G(1):=\left\{S_1,\ldots, S_{N+1}\right\}$.

  Assume that we can construct  a partition $\left\{G_j(k)\right\}_{j=1}^{N+1}$ of
  $G(k):=\left\{S_1,\ldots, S_{N+k}\right\}$
 with pairwise disjoint sets in each $G_j(k),\,k\ge 1,\,1\le j\le N+1$.
 Recall that $S_{N+k+1}$ has
 the nonempty intersection with no more than $N$ sets from $G(k)$.
 Therefore by the pigeonhole principle, there exists $G_{j_0}(k),\,1\le j_0\le N+1$,
 whose elements have the empty intersection with $S_{N+k+1}$. Finally, setting
 \ba
 G_j(k+1):=\left\{\begin{array}{ll}
 G_j(k), &j\ne j_0,\\
 G_{j_0}(k)\cup \left\{S_{N+k+1}\right\}, &j= j_0,
 \end{array}\right., 1\le j\le N+1,
 \ea
 we obtain the needed partition
 $\left\{G_j(k+1)\right\}_{j=1}^{N+1}$ of
  $G(k+1):=\left\{S_1,\ldots, S_{N+k+1}\right\}$.

  If $G$ is finite, then the inductive process stops after
  $\mathrm{card}(G)-N$ steps.
  If $G$ is countable, then setting
  $G_j=\bigcup_{k=1}^\iy G_j(k),\,1\le j\le N+1,$ we arrive at
  statement (a).
  In particular, the construction shows that certain subfamilies
  $G_j,\,1\le j\le N+1,$ could be empty only in case
  $\mathrm{card}(G)< N+1$.
  \vspace{.12in}\\
  (b) Note first that if $Q_\de^m(X_\mu)\in G$ and
  $Q_\de^m(X_\nu)\in G$ for $\nu\ne\mu$,
  then $ Q_\de^m(X_\mu)\cap Q_\de^m(X_\nu)\ne \emptyset$
  if and only if
  $\de_1\le \left\|X_\mu-X_\nu\right\|_{\iy}\le 2\de$
  (recall that $\de_1<2\de$).
  Indeed, if $x\in Q_\de^m(X_\mu)\cap Q_\de^m(X_\nu)$, then
  $\left\|X_\mu-X_\nu\right\|_{\iy}\le 2\de$ by the triangle inequality.
  If $\left\|X_\mu-X_\nu\right\|_{\iy}\le 2\de$,
  then the intervals $\left[X_{\mu,j}-\de,X_{\mu,j}+\de\right]$ and
  $\left[X_{\nu,j}-\de,X_{\nu,j}+\de\right]$ have at least one joint point
  $x_j,\,1\le j\le m$. Thus $x\in Q_\de^m(X_\mu)\cap Q_\de^m(X_\nu)$.

  Hence the number $N_1$ of cubes $Q_\de^m(X_\nu)\in G$ that have
  the nonempty intersection with a fixed $Q_\de^m(X_\mu)\in G,\,\nu\ne\mu$,
  can be estimated by the maximal number $N_2$ of points
  $X_\nu\in Q^\prime:=Q_{2\de}^m(X_\mu)\setminus \mathring{Q}_{\de_1}^m(X_\mu)$
  with the mutual "distance" of $\de_1$.
  Let us also set $Q^{\prime\prime}:=Q_{2\de}^m(X_\mu)\setminus \mathring{Q}_{\de_1/2}^m(X_\mu)$.
Since
  $\Omega$ is a $\de_1$-packing net for $\R^m$,
  the sets $\mathring{Q}_{\de_1/2}^m(X_\nu)\cap Q^{\prime\prime}$
  are pairwise disjoint for
  $X_\nu\in Q^\prime\cap \Omega$.
  In addition, if $x\in Q^\prime$, then
  $\left\vert \mathring{Q}_{\de_1/2}^m(x)
  \cap Q^{\prime\prime}\right\vert_m\ge(\de_1/2)^m$.
  Taking account of these two facts, we obtain
  \ba
  N_1\le N_2\le \frac{\left\vert Q^{\prime\prime}\right\vert_m}
  {\inf_{x\in Q^\prime}
  \left\vert \mathring{Q}_
  {\de_1/2}^m(x)\cap Q^{\prime\prime}\right\vert_m}
  \le 2^m\left(\left(\frac{4\de}{\de_1}\right)^m-1\right).
  \ea
  Thus statement (b) is established.
  \vspace{.12in}\\
  (c) The statement immediately follows from
  statements (a) and (b).
  \end{proof}
  Different versions of Lemma \ref{L2.2} were obtained by
  Brudnyi and Kotlyar \cite{BK1970}
  and Dol'nikov \cite[Theorem 1]{D1971}.

\subsection{Properties of Polynomials}\label{S2.2n}
We first prove a general result that reduces numerous
Markov-Bernstein-Nikolskii type inequalities for
complex-valued functions to real-valued ones.

\begin{lemma}\label{L2.2a}
For a compact set $K\subset\R^m$,
let $B$ be a subspace of $C(K)$
with a basis of real-valued functions
and $B_{\R}:=B\cap C_{\R}(K)$.
Next, let $L$
 be a bounded linear operator on
$B$ such that $L:B\to C(K)$
and $L:B_{\R}\to C_{\R}(K)$.
 In addition, let $\|\cdot\|$ be
 a monotone norm on $B$
(i.e., if $g\in B,\,h\in B$, and
$\vert g(x)\vert\le \vert h(x)\vert$ for $x\in K$,
then $\|g\|\le \|h\|$). Then
\beq \label{E2.2a}
\sup_{h\in B\setminus\{0\}}
\frac{\|L(h)\|_{C(K)}}{\|h\|}
=\sup_{g\in B_{\R}\setminus\{0\}}
\frac{\|L(g)\|_{C_{\R}(K)}}{\|g\|}.
\eeq
\end{lemma}
\begin{proof}
Every nonzero element $h\in B$ can be represented
 in the form $h=h_1+ih_2$, where $h_j\in B_{\R},\,j=1,2$.
 Then there exists $x_0\in K$ such that
 $\|L(h)\|_{C(K)}=\vert L(h)(x_0)\vert$.
 Assuming that $L(h)(x_0)\ne 0$,
 let us define $\g\in[0,2\pi)$ by the equality
 $e^{i\g}=L(h)(x_0)/\vert L(h)(x_0)\vert$.
 Then the element
 $g:=\cos \g \,h_1+\sin \g \,h_2$ belongs to $B_{\R}$ and
 satisfies the relations
 \ba
 \vert g(x)\vert\le \vert h(x)\vert,\quad x\in K;
 \qquad \vert L(g)(x_0)\vert=\vert L(h)(x_0)\vert.
 \ea
 Hence
 \beq\label{E2.2b}
 \frac{\|L(h)\|_{C(K)}}{\|h\|}
 \le \frac{\vert L(g)(x_0)\vert}{\|g\|}
 \le \frac{\|L(g)\|_{C_\R(K)}}{\|g\|}.
 \eeq
 Thus \eqref{E2.2a} is established.
 \end{proof}
 In particular, a linear operator
 in Lemma \ref{L2.2a} can be replaced
 by a linear functional; in this case
 the norms $\left\|\cdot\right\|_{C(K)}$ and
 $\left\|\cdot\right\|_{C_\R(K)}$ in \eqref{E2.2a}
 and \eqref{E2.2b}
  can be replaced by $\vert\cdot\vert$.
 Special cases of Lemma \ref{L2.2a}
 for differential operators $L$
 and for various
 sets $B$ of EFETs and
 algebraic or trigonometric polynomials
   were discussed in
  \cite[p. 567]{RS2002},  \cite[p. 30]{EGN2015},
  \cite[Theorem 1.1]{GT2017},
  \cite[Remark 4.2]{G2023}, and others.
  In the following two lemmas we apply Lemma \ref{L2.2a}
  to algebraic polynomials and special
  linear functionals $L$.
Two properties of the Chebyshev polynomial are discussed
below.

\begin{lemma}\label{L2.3}
(a) For any $u\in\R^1\setminus [-b,b],\,b>0$,
and $P_n\in\PP_n$,
\beq \label{E2.3}
\left\vert P_n^{(l)}(u)\right\vert
\le b^{-l}\left\vert T_n^{(l)}(u/b)\right\vert
\left\|P_n\right\|_{C([-b,b])},
\qquad l=0,\ldots,n.
\eeq
(b) For any $u\ge 1$ and $n\in\N$,
\beq \label{E2.4}
T_n(u)\le 2^{n-1} u^n.
\eeq
\end{lemma}
\begin{proof}
Statement (a) is well-known for
polynomials with real coefficients and $b=1$
(see, e.g., \cite[Eqn. (2.37)]{R1990}).
Inequality \eqref{E2.3} for a real-valued $P_n$
 follows from this case by a linear substitution.
 If $P_n$ is a complex-valued polynomial,
 then \eqref{E2.3} follows
 from Lemma \ref{L2.2a} for
 $K=[-b,b],\,B=\PP_n,\,\|\cdot\|
 =\|\cdot\|_{C([-b,b])}$, and
 the linear functional
 $L(h)(y):=h^{(l)}(u)$, where $h\in \PP_n,\,y\in [-b,b]$, and
 $u\in\R^1\setminus [-b,b]$ is a fixed point.

To prove \eqref{E2.4}, we set
$u=1/\sin \be,\,\be\in (0,\pi/2]$. Then
for $n\ge 1$,
\ba
2u^{-n}T_n(u)=(1+\cos \be)^n+(1-\cos \be)^n
=2^n\left(\cos^{2n}(\be/2)+\sin^{2n}(\be/2)\right)\le 2^n.
\ea
\end{proof}
Two multivariate versions of Lemma \ref{L2.3} (a)
are presented
in the following lemma.
\begin{lemma}\label{L2.4}
(a) For any $x\in\R^m\setminus V$ and $P\in\PP_{n,m}$,
\beq \label{E2.5}
\left\vert P(x)\right\vert
\le T_n\left(\frac{2\vert x\vert}{w(V)}\right)
\left\|P\right\|_{C(V)}.
\eeq
(b)  For any $\la>0,\,k\in\Z^m_+,\,x\in\R^m$ with
$\left\vert x_1\right\vert>\la,\ldots,
\left\vert x_m\right\vert>\la$,
and $P\in\QQ_{n,m}$,
\beq \label{E2.6}
\left\vert D^kP(x)\right\vert
\le \la^{-\langle k \rangle}
\left\vert T_n^{\left(k_1\right)}\left(x_1/\la\right)
\ldots T_n^{\left(k_m\right)}\left(x_m/\la\right)\right\vert
\left\|P\right\|_{C\left(Q^m_\la\right)}.
\eeq
\end{lemma}
\begin{proof}
(a) The restriction of $P$ to any straight line $\LL$,
passing through the origin,
is a univariate polynomial $P_n\in\PP_n$.
Then applying
\eqref{E2.3} for $l=0$ to $P_n$, we have
\ba
\left\vert P(x)\right\vert
\le \max\left\{\left\vert P_n(\vert x\vert)\right\vert,
\left\vert P_n(-\vert x\vert)\right\vert\right\}
&\le& \left\vert T_n\left(\frac{\pm 2\vert x\vert}
{\vert V\cap \LL\vert_1}\right)\right\vert
\left\|P_n\right\|_{C( V\cap \LL)}\\
&\le& T_n\left(\frac{2\vert x\vert}{w(V)}\right)
\left\|P\right\|_{C(V)}.
\ea
Thus \eqref{E2.5} holds.\vspace{.12in}\\
(b) We first apply
 Lemma \ref{L2.2a} for
 $K=Q^m_\la,\,B=\QQ_{n,m},\,\|\cdot\|
 =\|\cdot\|_{C\left(Q^m_\la\right)}$, and
 the linear functional
 $L(h)(y):=D^k h(x)$,
where $h\in \QQ_{n,m},\,
y\in Q^m_\la$, and
 a fixed point $x$ satisfies the conditions of statement (b).
 Therefore, it suffices to prove \eqref{E2.6} for real-valued
 polynomials $P$.

To prove the statement for real-valued polynomials,
 we need the following proposition
whose proof by induction is simple and left as an exercise to the reader.

\begin{proposition}\label{P2.5}
If $\vphi\in\PP_n$ is a polynomial with real coefficients
 and all zeros of $\vphi$ lie in $(-1,1)$, then
for any $u>1$ and any
$d=0,\ldots, n,\,\mathrm{sgn}\, \vphi^{(d)}(u)=\mathrm{sgn}\, \vphi(1)$.
\end{proposition}
Setting $\la=1$ for simplicity, we prove statement (b) by contradiction.
If the statement is invalid, then there exist a multi-index
$k\in  \Z^m_+\cap Q^m_n$, a point
$x^*=\left(x_1^*,\ldots, x_m^*\right)\in \R^m$
with $ x_1^*>1,\ldots,
 x_m^*>1$, and
a polynomial $P^*\in\QQ_{n,m}$
with $\left\|P^*\right\|_{C\left(Q^m_1\right)}\le 1$
such that
\beq \label{E2.7}
D^k P^*(x^*)
=M T_n^{\left(k_1\right)}\left(x_1^*\right)
\ldots T_n^{\left(k_m\right)}\left(x_m^*\right),
\eeq
where $M>1$.
Let us set
\beq \label{E2.8}
U(x):=P^*(x)-M T_n(x_1)\ldots T_n(x_m),\qquad x\in\R^m.
\eeq
Then $U\in \QQ_{n,m}$ and
\beq \label{E2.9}
D^k U(x^*)=0
\eeq
by \eqref{E2.7}.
We prove below that
\beq \label{E2.10}
D^k U(x^*)<0
\eeq
which contradicts \eqref{E2.9}.

To prove \eqref{E2.10}, we first note that for all
$l\in\Z^m_+\cap Q^m_n$,
\beq \label{E2.11}
U\left(\cos \frac{l_1\pi}{n},\ldots, \cos \frac{l_m\pi}{n}\right)
= P^*\left(\cos \frac{l_1\pi}{n},\ldots, \cos \frac{l_m\pi}{n}\right)
+(-1)^{1+\sum_{j=1}^m l_j} M
\eeq
by \eqref{E2.8}.
Then the polynomial
$\vphi_{0,n}(x_1)
:=U\left(x_1,\cos \frac{l_2\pi}{n},\ldots, \cos \frac{l_m\pi}{n}\right)$
belongs to $\PP_n$ and satisfies the condition
$\mathrm{sgn}\left(\vphi_{0,n}\left(\cos \frac{l_1\pi}{n}\right)\right)
=(-1)^{1+\sum_{j=1}^m l_j}$
by \eqref{E2.11}, since $\left\|P^*\right\|_{C(Q^m_1)}\le 1$ and
$M>1$.
Therefore, $\vphi_{0,n}$ changes its signs at points
$x_1=\cos \frac{l_1\pi}{n},\,0\le l_1\le n$,
and $\mathrm{sgn}\left(\vphi_{0,n}\left(1\right)\right)
=(-1)^{1+\sum_{j=2}^m l_j}$.
Moreover, since all zeros of $\vphi_{0,n}$ lie in $(-1,1)$, we see that
$\mathrm{sgn}\left(\vphi^{(k_1)}_{0,n}\left(x_1^*\right)\right)
=(-1)^{1+\sum_{j=2}^m l_j}$ by Proposition \ref{P2.5}.
 Thus
\beq \label{E2.12}
\mathrm{sgn}\left(\frac{\partial^{k_1}}{\partial x_1^{k_1}}
U\left(x_1^*,\cos \frac{l_2\pi}{n},\ldots, \cos \frac{l_m\pi}{n}\right)\right)
=(-1)^{1+\sum_{j=2}^m l_j}.
\eeq
Assume that for a fixed $p\in\N,\,1\le p\le m-1$,
\beq \label{E2.13}
\mathrm{sgn}\left(
D^k U\left(x_1^*,\ldots,x_p^*,\cos \frac{l_{p+1}\pi}{n},\ldots,
 \cos \frac{l_m\pi}{n}\right)\right)
=(-1)^{1+\sum_{j=p+1}^m l_j}.
\eeq
For $p=1$ \eqref{E2.13} is valid by \eqref{E2.12}.
Then assumption \eqref{E2.13} shows that
the polynomial
\ba
\vphi_{p,n}\left(x_{p+1}\right)
:=
D^k U\left(x_1^*,\ldots,x_p^*,x_{p+1},\cos \frac{l_{p+2}\pi}{n},\ldots,
\cos \frac{l_m\pi}{n}\right)
\ea
from $\PP_n$
changes its signs at points
$x_{p+1}=\cos \frac{l_{p+1}\pi}{n},\,0\le l_{p+1}\le n$,
and
$
\mathrm{sgn}\left(\vphi_{p,n}\left(1\right)\right)=
\linebreak
(-1)^{1+\sum_{j=p+2}^m l_j}.
$
Moreover, since all zeros of $\vphi_{p,n}$ lie in $(-1,1)$, we see that
$
\mathrm{sgn}\left(\vphi^{(k_{p+1})}_{p,n}\left(x_{p+1}^*\right)\right)=
\linebreak
(-1)^{1+\sum_{j=p+2}^m l_j}
$
 by Proposition \ref{P2.5}. Therefore, \eqref{E2.13} holds true
with $p$ replaced by $p+1$. Thus by induction, \eqref{E2.13} is valid for every $p\in\N,\,
1\le p\le m$. In particular, for $p=m$ \eqref{E2.13} is equivalent to \eqref{E2.10}.
This proves inequality \eqref{E2.6}.
\end{proof}
Different versions of Lemma \ref{L2.4} (a) were discussed by Rivlin and Shapiro
\cite[Problem 3]{RS1961} and by Brudnyi and the author \cite[Eqn. ($2^\prime$)]{BG1973}.
In addition, Lemma \ref{L2.4} (b) for $k=0$
 was proved by Bernstein \cite[Theorem 2]{B1948a}. The proof of statement (b)
 is based on an idea from \cite{B1948a}.

 The following property is a corollary of Lemma \ref{L2.4} (b).
 \begin{lemma}\label{L2.6}
 Let $U(y)=\sum_{k\in\Z^m_+\cap Q^m_n}c_k\,
  y^k\in\QQ_{n,m}$.\\
 (a) If $[A,B]\subset (0,\iy)$, then
 \beq \label{E2.14}
 \sum_{k\in\Z^m_+\cap Q^m_n}\left\vert c_k\right\vert
\le \left[T_n\left(\frac{B+A+2}{B-A}
\right)\right]^m\| U\|_{C\left([A,B]^m\right)}.
 \eeq
 (b) If $b>0$, then
 \beq \label{E2.15}
 \sum_{k\in\Z^m_+\cap Q^m_n}\left\vert c_k\right\vert
 \le\left(\frac{e^{b/4}+e^{-b/4}}{e^{b/4}-e^{-b/4}}\right)^{mn}
 \|U\|_{C\left(\left[e^{-b},e^b\right]^m\right)}.
 \eeq
 \end{lemma}
 \begin{proof}
 (a) Applying Lemma \ref{L2.4} (b) to the polynomial
 $P(x):=U\left(\frac{B+A}{2}-x_1,\ldots,\frac{B+A}{2}-x_m\right) \in \QQ_{n,m}$
 for $\la=(B-A)/2$ and $x=\left((B+A)/{2},\ldots,(B+A)/2\right)$,
 we obtain
 \bna \label{E2.16}
 \sum_{k\in\Z^m_+\cap Q^m_n}\left\vert c_k\right\vert
 &=& \sum_{k\in\Z^m_+\cap Q^m_n}
 \frac{1}{\prod_{j=1}^m k_j!}
 \left\vert D^k P(x)\right\vert_{x_1=\ldots=x_m=(B+A)/2}
\nonumber\\
 &\le& \sum_{k\in\Z^m_+\cap Q^m(n)}
 \prod_{j=1}^m \frac{1}{k_j!}\left(\frac{B-A}{2}\right)^{-k_j}
 T_n^{\left(k_j\right)}\left(\frac{B+A}{B-A}\right)
 \left\|P\right\|_{C\left(Q^m_{(B-A)/2}\right)}\nonumber\\
 &=& \left[T_n\left(\frac{B+A}{B-A}+\frac{2}{B-A}\right)\right]^m
 \left\|U\right\|_{C\left([A,B]^m\right)}.
 \ena
 Note that the last equality in \eqref{E2.16} follows from Taylor's formula.
 Thus \eqref{E2.14} is established.\\
  (b) Since
 \ba
 T_n\left(\frac{e^b+e^{-b}+2}{e^b-e^{-b}}\right)
 =T_n\left(\frac{e^{b/2}+e^{-b/2}}{e^{b/2}-e^{-b/2}}\right)
 =\frac{1}{2}\left( \left(\frac{e^{b/4}+e^{-b/4}}{e^{b/4}-e^{-b/4}}\right)^n
 +\left(\frac{e^{b/4}-e^{-b/4}}{e^{b/4}+e^{-b/4}}\right)^n\right),
 \ea
 inequality \eqref{E2.15} follows from \eqref{E2.14}.
 \end{proof}

 \subsection{Properties of EFETs}\label{S2.3n}
 We first need several Bernstein-type inequalities.
  \begin{lemma}\label{L2.7}
  (a) If $f\in B_{\sa,m}\cap L_q(\R^m),\,q\in [1,\iy]$, then
  \beq\label{E2.17a}
  \left\|D^k f\right\|_{L_q(\R^m)}
  \le \sa^{\langle k\rangle}\|f\|_{L_q(\R^m)}.
  \eeq
  (b)  If $f\in B_{\sa,m}\cap
  L_q\left(\R^m\right),
  \,q\in [1,\iy),\,d\in\N$, and $h>0$, then
  \bna\label{E2.19}
  &&I_h(f):=
  \left(\sum_{r\le\langle k\rangle\le d+r}h^{\langle k\rangle q}\left\|D^ k f\right\|_
  {L_q\left(\R^m\right)}^q\right)^{1/q}\nonumber\\
  &&\le \left(\binom{m+d+r}{m}-r\right)^{1/q}
  \max\left\{(h\sa)^r,(h\sa)^{d+r}\right\}
  \|f\|_{L_q\left(\R^m\right)},\qquad r=0,\,1.
  \ena
  (c)  If $f\in B_{\sa,m}\cap L_\iy(\R^m)$, then
  \beq\label{E2.18}
  \left\|\sum_{j=1}^m\left\vert\frac{\partial f(x)}{\partial x_j}
  \right\vert\right\|_{L_\iy(\R^m)}
  \le m\sa\|f\|_{L_\iy(\R^m)}.
  \eeq
  \end{lemma}
  \begin{proof}
  Statement (a) is a multivariate version of Bernstein's inequality
  (see, e.g., \cite[Eqn. 3.2.2(8)]{N1969}) and
  (c) follows directly from (a).
  Next, using \eqref{E2.17a}, we have
  \beq\label{E2.18a}
  I_h(f)
  \le \left(\sum_{l=r}^{d+r}(h\sa)^{l q}\binom{m+l-1}{m-1}\right)^{1/q}
  \|f\|_{L_q(\R^m)}.
  \eeq
  Then statement (b) follows from  \eqref{E2.18a} and the known identity
  \beq\label{E2.18b}
  \sum_{0\le l\le d+r} \binom{m+l-1}{m-1}=\binom{m+d+r}{m},
  \eeq
  where the right-hand side of \eqref{E2.18b} coincides with the dimension of the space $\PP_{d+r,m}$
  (see, e.g., \cite[Eqn. (3.8)]{R2003}) .
  \end{proof}
  Note that more general
  inequalities than \eqref{E2.18} were recently  proved in
  \cite[Theorem 2.1 and Corollary 2.4]{G2023}.
Next, we discuss  a certain technical discretization inequality
  that plays an important role in the proof of Theorem \ref{T1.1}.
  \begin{lemma}\label{L2.8}
  Let $\left\{Q_h^m\left(X_\nu\right)\right\}_{\nu=1}^\iy$
  be a family of closed cubes with the pairwise disjoint interiors
  and let $Y_\nu\in Q_h^m\left(X_\nu\right),\,\nu\in\N$.
  In addition, let $d$ be defined by \eqref{E1.3.01}.
  If $f\in B_{\sa,m}\cap L_q(\R^m),\,q\in [1,\iy)$, then
  \beq \label{E2.20}
  \left\vert \left(\sum_{\nu=1}^\iy\left\vert f
  \left(X_\nu\right)\right\vert^q\right)^{1/q}
  -\left(\sum_{\nu=1}^\iy\left\vert f
  \left(Y_\nu\right)\right\vert^q\right)^{1/q}\right\vert
  \le C(m,q)h^{-m/q}\max\left\{h\sa,(h\sa)^{d}\right\}
  \|f\|_{L_q(\R^m)}.
  \eeq
  \end{lemma}
  \begin{proof}
  We first prove
  two estimates for the modulus of continuity
  of differentiable functions on a cube.
Let $F$ be a $d+1$ times continuously differentiable
  function on $Q^m_h\left(X_0\right)$, where $X_0\in\R^m$.
  Then for any
  $X\in Q^m_h\left(X_0\right)$ and $Y\in Q^m_h\left(X_0\right)$,
  the following estimate is valid
  by the Sobolev embedding theorem with $dq>m$ (see, e.g., \cite[Theorem 5.4, Eqn. (8)]{A1975}):

  \beq\label{E2.20a}
  \left\vert F(X)-F(Y)\right\vert
  \le 2\|F\|_{C\left(Q^m_h\left(X_0\right)\right)}
  \le C(m,q)h^{-m/q}
  \left(\sum_{0\le\langle k\rangle\le d}h^{\langle k\rangle q}\left\|D^k F\right\|_
  {L_q\left(Q^m_h\left(X_0\right)\right)}^q\right)^{1/q}.
  \eeq
  On the other hand, using again the embedding theorem, we have
  \bna\label{E2.20b}
  \left\vert F(X)-F(Y)\right\vert
  &\le& \sqrt{m}\vert X-Y\vert \sup_{\langle k\rangle=1}
  \left\|D^k F\right\|_{C\left(Q^m_h\left(X_0\right)\right)}\nonumber\\
  &\le& C(m,q)h^{-m/q}
  \left(\sum_{1\le\langle k\rangle\le d+1}h^{\langle k\rangle q}\left\|D^k F\right\|_
  {L_q\left(Q^m_h\left(X_0\right)\right)}^q\right)^{1/q}.
  \ena

  Next, using Minkowski's inequality for sums and estimates \eqref{E2.20a} and \eqref{E2.20b}
  for $F=f$, we obtain
  \bna\label{E2.20c}
  &&\left\vert \left(\sum_{\nu=1}^\iy\left\vert f
  \left(X_\nu\right)\right\vert^q\right)^{1/q}
  -\left(\sum_{\nu=1}^\iy\left\vert f
  \left(Y_\nu\right)\right\vert^q\right)^{1/q}\right\vert
  \le \left(\sum_{\nu=1}^\iy\left\vert f
  \left(X_\nu\right)-f
  \left(Y_\nu\right)
  \right\vert^q\right)^{1/q}\nonumber\\
  &&\le C(m,q)h^{-m/q}
  \min_{r\in\{0,1\}}
  \left(\sum_{r\le\langle k\rangle\le d+r}h^{\langle k\rangle q}\left\|D^k F\right\|_
  {L_q\left(\R^m\right)}^q\right)^{1/q}.
  \ena
Finally, inequality \eqref{E2.20} follows from \eqref{E2.20c} and
Bernstein-type inequality  \eqref{E2.19}.
\end{proof}

\section{Approximation of EFETs by Polynomials and Entire Functions}
\label{S3n}
\setcounter{equation}{0}
\noindent
In this section we discuss approximation of univariate EFETs by polynomials
on compacts from $\CC$ and approximation of multivariate EFETs by polynomials
on the octahedron and the cube. In addition, we also study unconventional approximation of
 EFETs by other EFETs.
\subsection{Approximation of EFETs by Polynomials}\label{S3.1n}
Approximation of EFETs by algebraic polynomials was
initiated by Bernstein \cite{B1946,B1947} and independently
it was
discussed by Logvinenko \cite{L1974,L1975}.
Various univariate and multivariate versions
of these results were obtained
by the author \cite{G1982,G1991,G2019b}. Most of
approximation theorems from these publications have discussed
EFETs either bounded or of polynomial growth on $\R^m$.
However, in this paper we need estimates of univariate and
mutivariate polynomial approximation
for general EFETs. Some of these results are based on
estimates of Chebyshev coefficients.
\subsubsection{Estimates of Chebyshev coefficients}\label{S3.1.1n}
Let $\sum_{k\in\Z^m_+}c_{k,m}(f,b)\prod_{j=1}^m T_{k_j}\left(x_j/b\right),
\,x\in Q^m_b,\,b>0,$
be the multivariate Fourier-Chebyshev series of a function
$f\in L_\iy\left(Q^m_b\right)$ with the coefficients
 \bna\label{E3.1}
 c_{k,m}(f,b)&:=&\frac{(2/\pi)^m}{2^{r_m(k)}}
 \bigintss_{Q^m_b}f(x)
 \prod_{j=1}^m\frac{ T_{k_j}\left(x_j/b\right)}{\sqrt{b^2-x_j^2}}dx\nonumber\\
 &=&\frac{(2/\pi)^m}{2^{r_m(k)}}
 \int_{[0,\pi]^m}f\left(b\cos t_1,\ldots,b\cos t_m\right)
 \prod_{j=1}^m \cos k_jt_j\,dt,\qquad k\in\Z^m_+.
 \ena
 Here, $r_m(k)$ is the number of zero components in a vector $k\in\Z^m_+$.

 In addition, let
 \beq\label{E3.2}
 \Gamma_R^m(b):=\left\{w=x+iy\in\CC^m:
 \left(\frac{x_j}{R/b+b/R}\right)^2
 +\left(\frac{y_j}{R/b-b/R}\right)^2\le(b/2)^2,\,1\le j\le m\right\}
 \eeq
 be the direct product of the sets encircled by the corresponding ellipses
  in $\CC^1$ with foci at the ends of $[-b,b]$ and with the sum of its semi-axes
  equal to $R>b$.
\begin{lemma}\label{L3.1}
If $f$ is a holomorphic and bounded function on the interior of $\Gamma_R^m(b)$,
then
 \beq\label{E3.3}
 \left\vert c_{k,m}(f,b)\right\vert\le 2^{m-r_m(k)}(b/R)^{\langle k\rangle}
 \|f\|_{L_\iy\left( \Gamma_R^m(b)\right)},\qquad k\in\Z^m_+.
 \eeq
\end{lemma}
\begin{proof}
Assume for simplicity that $b=1$, and let $f$ satisfy the conditions of the lemma.
Note that the univariate estimate
\beq\label{E3.4}
\left\vert c_{k_1,1}(f,1)\right\vert\le 2^{1-r_1(k_1)}R^{-k_1}
 \|f\|_{L_\iy\left( \Gamma_R^1(1)\right)},\qquad k_1\in\Z^1_+,
 \eeq
 is well known (see \cite[Sect. 2.1, Lemma 1]{B1937} and \cite[Sect. 3.7.3]{T1963}).
 Let us set
 \ba
 \vphi_l\left(z_1,\ldots,z_l\right):=f\left(z_1,\ldots,z_l,z_{l+1},\ldots,z_m\right)
 \ea
 with fixed parameters $z_{l+1},\ldots,z_m,\,1\le l\le m$.
 To prove the lemma, it suffices to establish the inequality
 \beq\label{E3.5}
 \left\vert c_{k,p}\left(\vphi_p,1\right)\right\vert
 \le 2^{p-r_p(k)}R^{-\langle k\rangle}
 \left\|\vphi_p\right\|_{L_\iy\left( \Gamma_R^p(1)\right)},\qquad k\in\Z^p_+,
 \quad 1\le p\le m,
 \eeq
 by induction in $p$. For $p=1$ \eqref{E3.5} follows from \eqref{E3.4}. Next, assume
 that \eqref{E3.5} is valid for $p=l,\,k\in\Z^l_+,\,1\le l\le m-1$.
 Since $ c_{k,l}\left(\vphi_l,1\right)$
  holomorphic and bounded function in $z_{l+1}$ on the interior of $\Gamma_R^1(1)$,
  we can use \eqref{E3.4} and also \eqref{E3.5} for $p=l$ to obtain
  the following relations for $k\in\Z^{l+1}_+$ and $k(l):=\left(k_1,\ldots,k_l\right)$,
  \ba
  \left\vert c_{k,l+1}\left(\vphi_{l+1},1\right)\right\vert
  &=&\left\vert\bigintss_{-1}^1 \frac{c_{k(l),l}\left(\vphi_{l},1\right)T_{k_{l+1}}(x_{l+1})}
  {\sqrt{1-x^2_{l+1}}}\,dx_{l+1}\right\vert\\
  &\le& 2^{1-r_1\left(k_{l+1}\right)}R^{-k_{l+1}}\sup_{z_{l+1}\in \Gamma_R^1(1)}
  \left\vert c_{k_{(l)},l}\left(\vphi_{l},1\right)\right\vert\\
  &\le& 2^{l+1-r_{l+1}(k)}R^{-\langle k\rangle}\sup_{z_{l+1}\in \Gamma_R^1(1)}
  \left\|\vphi_l\right\|_{L_\iy\left( \Gamma_R^l(1)\right)}.
  \ea
  Therefore, \eqref{E3.5} holds true for $p=l+1$. Thus \eqref{E3.3} is valid.
\end{proof}
\subsubsection{Univariate approximation.}\label{S3.1.2n}
Here, we discuss univariate approximation on a symmetric (with respect to the origin)
compact $K\subset\CC^1$.

Let $z(w)=w+\sqrt{w^2-1}$ denote the conformal map of the ellipse $\Gamma_R^1(1)$
defined in \eqref{E3.2} onto the circle $\{z\in\CC^1:\vert z\vert=R\}$.
Let us set
\beq\label{E3.6}
\al=\al(K):=\max_{w\in K}\left\vert w+\sqrt{w^2-1}\right\vert
=\min_{R\in(1,\iy)}\{R:K\subseteq \Gamma^1_R(1)\}
=\left\vert w_0+\sqrt{w_0^2-1}\right\vert,
\eeq
where $w_0\in K\cap \Gamma^1_{R_0}(1)$ and $R_0$ is the extremal value in \eqref{E3.6}.
Since $K$ is symmetric, the next estimate immediately follows from \eqref{E3.6} and \eqref{E1.1}:
\beq\label{E3.6a}
\max_{w\in bK}\vert T_k(w/b)\vert \le \al^k.
\eeq
In addition, let us set
\ba
\psi(\g,K):=\frac{\sqrt{1+\g^2}}{\g}
-\log\left(\frac{\g+\sqrt{1+\g^2}}{\al}\right)
=\psi(\g)+\log \al,
\qquad \g\in(0,\iy),
\ea
where $\psi(\g)=\psi(\g,[-1,1])$ is defined by \eqref{E1.3.5a}.
Since $\psi$ is a strictly decreasing function in $\g$
 on $(0,\iy)$, there exists the unique solution
$\g_0=\g_0(\al)\in(0,\iy)$ to the equation $\psi(\g,K)=0$, and, in addition,
$\psi(\g,K)<0$ for $\g>\g_0$ and $\g_0+\sqrt{1+\g_0^2}>\al$.

\begin{lemma}\label{L3.2}
Given $\sa>0,\,b>0,$ and $n\in\N$,
let us denote
$\tau:=\frac{n}{\sa b}$.
In addition, let
  $f\in B_\sa$ satisfy the condition
\beq\label{E3.7}
\vert f(\xi)\vert\le Ae^{\sa\vert \xi\vert},\qquad \xi\in\CC^1,
\eeq
where $A>0$ is a constant.
If $\tau>\g_0(\al)$, then there exists a polynomial $U_n\in\PP_n$ such that
\beq\label{E3.8}
\max_{w\in bK}\left\vert f(w)-U_n(w)\right\vert
=\max_{w\in \frac{n}{\sa \tau}K}\left\vert f(w)-U_n(w)\right\vert
\le C(K)Ae^{n\psi(\tau,K)}.
\eeq
\end{lemma}
\begin{proof}
It follows from \eqref{E3.7} that for any $\de>0$,
\beq\label{E3.9}
\vert f(w)\vert\le Ae^{\sa b\sqrt{1+\de^2}}, \qquad w\in\Gamma^1_R(b),
\eeq
where $R=b\left(\de+\sqrt{1+\de^2}\right)$.

Next, following Bernstein \cite{B1947} (see also \cite[Sect. 5.4.4]{T1963}),
we approximate $f$ by the partial Fourier-Chebyshev sum
$U_n(w):=\sum_{k=0}^nc_{k,1}T_k(w/b)$, where
$c_{k,1}=c_{k,1}(f,b),\,k\in\Z^1_+$, is defined
by \eqref{E3.1}.
Then the estimate
\beq\label{E3.10}
\left\vert c_{k,1}\right\vert
\le \frac{2Ae^{\sa b\sqrt{1+\de^2}}}{\left(\de+\sqrt{1+\de^2}\right)^k},
\qquad k\in\Z^1_+,
\eeq
follows from \eqref{E3.9} and \eqref{E3.3} for $m=1$ (see also \eqref{E3.4}).

Since $f$ is an entire function, we see that
$f(w)=\sum_{k=0}^\iy c_{k,1}T_k(w/b)$ for $w\in \Gamma_R^1(b)$
(see \cite[Theorem 9.1.1]{S1959}).
Therefore, if $\al<\de+\sqrt{1+\de^2}$, then using \eqref{E3.10} and \eqref{E3.6a}, we obtain

\bna\label{E3.11}
&&\max_{w\in bK}\left\vert f(w)-U_n(w)\right\vert
\le \sum_{k=n+1}^\iy \left\vert c_{k,1}\right\vert\al^k\nonumber\\
&&\le \frac{2A}{1-\al/(\de+\sqrt{1+\de^2})}
\exp\left[\sa b\sqrt{1+\de^2}
-n\log\left(\frac{\de+\sqrt{1+\de^2}}{\al}\right)\right].
\ena
Note that if $\de=\tau=\frac{n}{\sa b}>\g_0(\al)$, then
$\al<\g_0+\sqrt{1+\g_0^2}< \de+\sqrt{1+\de^2}$. Then choosing
$\de=\tau$ in \eqref{E3.11}, we arrive at \eqref{E3.8} with
\ba
C(K)\le \frac{2}{1-\al/\left(\g_0+\sqrt{1+\g_0^2}\right)}.
\ea
\end{proof}
Examples of $K,\,\al,\,w_0$, and $\g_0(\al)$ are given below.
\begin{example}\label{Ex3.3}
(a) $K=[-1,1],\,\al=1,\,\g_0(\al)=1.5088\ldots$;\\
(b) $K=\Gamma_R^1(1),\,\al=R;$\\
(c) $K=\left\{w\in\CC^1:\vert w\vert\le M\right\},\,\al=M+\sqrt{M^2+1},\,
w_0=iM$;\\
(d) $K=\left\{w\in\CC^1:\vert w\vert\le 1\right\},\,\al=1+\sqrt{2},\,
\g_0(\al)=3.3541\ldots;$\\
(e) $K=\left\{x+iy\in\CC^1:\vert x\vert\le 1,\,\vert y\vert\le 1\right\},\,
\al=\frac{1+\sqrt{5}}{2}+\sqrt{\frac{1+\sqrt{5}}{2}}=2.8900\ldots,\,
w_0=1+i,\,\g_0(\al)=3.9896\ldots$.
\end{example}
Example \ref{Ex3.3} (a) is trivial, while examples (b), (c), (d), and (e)
 follow from relations \eqref{E3.6}.

 \begin{remark}\label{R3.4}
 Lemma \ref{L3.2} will be used in the following two situations:\\
 \textbf{Case 1.} $\sa$ is independent of $n$ and $b$ is proportional to $n$.\\
 \textbf{Case 2.} $b$ is independent of $n$ and $\sa$ is proportional to $n$.

 In both cases $\tau>\g_0$ is a fixed number, so the right-hand side of
 \eqref{E3.8} is $o(1)$ as $n\to\iy$.
Case 2 has never been used before, while
 various versions of Case 1 have been discussed since the 1940s.
 For the interval $K$
 from Example \ref{Ex3.3} (a),
 a weaker version of \eqref{E3.8} was proved by Bernstein \cite{B1947}
 (see also \cite[Sect. 5.4.5]{T1963} and \cite[Appendix, Sect. 83]{A1965}).
 For the unit disk $K$
 from Example \ref{Ex3.3} (d), Logvinenko \cite[Lemmas 2]{L1974,L1975}
 established the relation
 $\lim_{n\to\iy}\left\|f-P_n\right\|_{L_\iy\left(\frac{n}{\sa \tau}K\right)}=0$
 with the Taylor polynomial $P_n$
 and an integer $\tau>e$, while inequality
 \eqref{E3.8} is valid only for $\tau>3.3541\ldots$.
 The author \cite[Lemma 4.5]{G1982} proved Lemma \ref{L3.2}
 for $f(w)=e^{\sa w}$ and
 the square $K$
 from Example \ref{Ex3.3} (e).
 \end{remark}
  \begin{remark}\label{R3.5}
   The following relation shows that for $K=[-1,1]$,
   a fixed $\sa>0$, and a fixed $\tau>\g_0(1)=1.5088\ldots$,
   estimate
  \eqref{E3.8} cannot be essentially improved:
  \ba
  \lim_{n\to\iy}\left(\inf_{U_n\in\PP_n}
  \max_{w\in \left[-\frac{n}{\sa \tau},\frac{n}{\sa \tau}\right]}
  \left\vert e^{\sa w}-U_n(w)\right\vert\right)^{1/n}
  =e^{\psi(\tau)}.
  \ea
The corresponding upper estimate follows from \eqref{E3.8},
while the lower one was proved in \cite[Appendix, Sect. 83]{A1965}.
\end{remark}

\subsubsection{Multivariate approximation.}\label{S3.1.3n}
Here, we discuss two multivariate versions of
Lemma \ref{L3.2}.
The first one is an extension of Case 1 for polynomials
from $\PP_{n,m}$
 and the second one is an extension of Case 2 for
 polynomials from
 $\QQ_{n,m}$.

\begin{lemma}\label{L3.6}
For $f\in B_{\sa,m}$ and for a fixed $\tau\ge 4$ there exists
 a polynomial
$P_n\in\PP_{n,m}$ such that
\beq\label{E3.12}
\left\|f-P_n\right\|_{L_\iy\left((n/\tau)O_{1/\sa}^m\right)}\le C_4 e^{-an},
\eeq
where $C_4=C_4(f,\tau,\sa,m)$ and $a=a(\tau)>0$ are
independent of $n$.
\end{lemma}
\begin{proof}
First, let us set
$W:=\left\{x+iy\in\CC^m:x\in Q_\sa^m,y\in Q_\sa^m\right\}$.
In addition, let $U_n\in\PP_n$ be a polynomial from
Lemma \ref{L3.2}
for $f(\xi)=e^\xi$, and let $K_1$ be the square from
Example \ref{Ex3.3} (e).

Then it follows from \eqref{E3.8} and Example \ref{Ex3.3}
 (e) that for any
$\vep>0$ and $\tau/(1+\vep)>3.9896\ldots$, the following
inequalities hold:
\bna\label{E3.13}
\max_{t\in (n/\tau)O_{1/\sa}^m}\max_{w\in (1+\vep)W}
\left\vert e^{(t,w)}-U_n((t,w))\right\vert
&\le& \max_{\xi\in ((1+\vep)n/\tau)K_1}
\left\vert e^{\xi}-U_n(\xi)\right\vert\nonumber\\
&\le& C(K_1)e^{n\psi(\tau/(1+\vep),K_1)}.
\ena
Next, for any $z=x+iy\in\CC^m$,
\ba
\sa\sum_{j=1}^m\left\vert z_j\right\vert
\le \sa\sum_{j=1}^m\left(\left\vert x_j\right\vert
+\left\vert y_j\right\vert\right)
=\sup_{w\in W}\mathrm{Re} (z,w)
:=H_W(z).
\ea
Hence for any $\vep>0$,
\ba
\vert  f(z)\vert \le C(f,\vep)e^{ H_W(z)+\vep\vert z\vert},
\ea
and by the Ehrenpreis-Martineau theorem
\cite{E1961, M1963, R1974}, there exists a continuous function
$\vphi_{\vep,f}$ on $\CC^m$, with $\vphi_{\vep,f}=0$ on $\CC^m\setminus (1+\vep)W$
such that
\beq\label{E3.14}
f(z)=\int_{(1+\vep)W} \vphi_{\vep,f}(w)e^{(z,w)}\,dw,\qquad z\in\CC^m,
\eeq
(see, e.g., the proof of representation \eqref{E3.14} in
\cite[Theorem 3.6.5]{R1974}).
Further, setting
\ba
P_n(x):=\int_{(1+\vep)W} \vphi_{\vep,f}(w)U_n((x,w))\,dw,
\ea
we see that $P_n\in\PP_{n,m}$, and it follows from \eqref{E3.13}
that
\ba
\left\|f-P_n\right\|_{L_\iy\left((n/\tau)O_{1/\sa}^m\right)}
\le C(K_1)\int_{(1+\vep)W} \left\vert\vphi_{\vep,f}(w)\right\vert\,dw
\,\,e^{n\psi(\tau/(1+\vep),K_1)}
\ea
for $\tau/(1+\vep)>\g_0=\g_0(\al(K_1))=3.9896\ldots$
(see  Example \ref{Ex3.3} (e)).
It remains to choose $\vep=3.9897/\g_0-1$ and to set
$a:=-\psi(\g_0\tau/3.9897,K_1)$. Then \eqref{E3.12} is valid
for $\tau\ge 4$.
\end{proof}

\begin{lemma}\label{L3.7}
Let $b>0,\,\tau>\g_0(1)=1.5088\ldots,\,A>0$, and
$n\in\N$ be given numbers.
In addition, let $f$ be an entire function, satisfying the inequality
\beq\label{E3.15}
\vert f(w)\vert\le A\exp\left(\sa\sum_{j=1}^m
\left\vert w_j\right\vert\right),\qquad w\in\CC^m,
\eeq
where $\sa={n}/{(mb\tau)}$.
Then there exists a polynomial $P_n\in\QQ_{n,m}$ such that
\beq\label{E3.16}
\left\|f-P_n\right\|_{L_\iy\left(Q^m_b\right)}\le C(m) A e^{n\psi(\tau)},
\eeq
where
$C\le m2^m\left(1-1/\left(\tau+\sqrt{1+\tau^2}\right)\right)^{-m}
<m2^m \left(1.44^{m}\right)$.
\end{lemma}
\begin{proof}
Note first that it follows from \eqref{E3.15} that for any $\de>0$,
\beq\label{E3.17}
\vert f(w)\vert
\le Ae^{m\sa b\sqrt{1+\de^2}},\qquad w\in\Gamma_R^m(b),
\eeq
(see \eqref{E3.9} for $m=1$), where
$R=b\left(\de+\sqrt{1+\de^2}\right)$
and $\Gamma_R^m(b)$ is defined by \eqref{E3.2}.

Similarly to the proof of Lemma \ref{L3.2}, we approximate $f$ by
the multivariate partial Fourier-Chebyshev sum
$\sum_{k\in Q^m_n\cap\Z^m_+}c_{k,m}(f,b)
\prod_{j=1}^m T_{k_j}\left(x_j/b\right)$,
where $c_{k,m}(f,b),\,k\in\Z^m_+,$ is defined in \eqref{E3.1}.
Then the following estimate for the Chebyshev coefficients follows from
\eqref{E3.17} and \eqref{E3.3}:
\beq\label{E3.18}
\left\vert c_{k,m}\right\vert
\le \frac{2^mAe^{m\sa b\sqrt{1+\de^2}}}
{\left(\de+\sqrt{1+\de^2}\right)^{\langle k\rangle}},
\qquad k\in\Z^m_+.
\eeq
Next,
\beq\label{E3.19}
f(w)=
\sum_{k\in\Z^m_+}c_{k,m}(f,b)\prod_{j=1}^m T_{k_j}\left(w_j/b\right),
\qquad w\in Q^m_b.
\eeq
Indeed, the series $\sum_{k\in\Z^m_+}c_{k,m}(f,b)\prod_{j=1}^m\cos k_jt_j$
converges uniformly on $Q^m_\pi$ to a function $S$\linebreak
 by estimate
\eqref{E3.18}.
Then the Fourier coefficients of $S$ coincide with $c_{k,m},\,k\in\Z^m_+,$
so $S(x)=f\left(b\cos x_1,\ldots,b\cos x_m\right)$ and
\eqref{E3.19} holds.

Finally, it follows from \eqref{E3.19} and \eqref{E3.18} that
\bna\label{E3.20}
\left\|f-P_n\right\|_{L_\iy\left(Q^m_b\right)}
&\le& \sum_{k\in\Z^m_+\setminus Q^m_n}\left\vert c_{k,m}(f,b)\right\vert
\nonumber\\
&\le& 2^mAe^{m\sa b\sqrt{1+\de^2}}
\left(\sum_{k\in\Z^m_+}
\left(\de+\sqrt{1+\de^2}\right)^{-\langle k\rangle}
-\sum_{k\in Q^m_n}\left(\de+\sqrt{1+\de^2}\right)^{-\langle k\rangle}
\right) \nonumber\\
&\le& m2^mA\left(1-1/\left(\de+\sqrt{1+\de^2}\right)\right)^{-m}
e^{m\sa b\sqrt{1+\de^2}-n\log\left(\de+\sqrt{1+\de^2}\right)}.
\ena
Choosing $\de=\tau$ in \eqref{E3.20}, we arrive at \eqref{E3.16}.
\end{proof}

\begin{remark}\label{R3.7a}
The following version of Lemmas \ref{L3.6} and \ref{L3.7} was
proved in \cite[Lemma 2]{L1974}:
if $f\in B_{1,m}$, then
$\lim_{n\to\iy}\left\|f-P_n\right\|_{L_\iy\left(Q^m_{n/\tau}\right)}=0$
 with the Taylor polynomial $P_n\in\PP_{n,m}$
 and an integer $\tau>\lceil em\rceil$.
It is difficult to compare this result with Lemmas \ref{L3.6} and \ref{L3.7}
 because the sets $O_{1}^m$ and $Q_{1}^m$,
 the conditions $\tau\ge 4$ and $\tau>\lceil em\rceil$,
 and the polynomial classes $\PP_{n,m}$ and $\QQ_{n,m}$
  are different.
\end{remark}

\subsection{Approximation of EFETs by Entire Functions}\label{S3.2n}
Throughout Section \ref{S3.2n}, $f\in B_{\sa,m},\,q\in[1,\iy]$, and
$n\in\N$; we also set
$w^*:=w\left(O_{1/\sa}^m\right)=2/\left(\sa\sqrt{m}\right)$ by \eqref{E1.0}.
Given $\vep>0$ and $\tau\ge 4$, let us set
\beq\label{E3.21}
\be=\be(\tau,\vep,w^*):=\frac{2\tau e^{(1+2\vep)/\tau}}{w^*}.
\eeq
In addition, let $P_{2n}\in\PP_{2n,m}$ be a polynomial from Lemma \ref{L3.6}.
We define a sequence of entire functions of spherical type
$2\be+O(1/n)$ (see Definition \ref{D1.3}) as $n\to\iy$ by the formula
\beq\label{E3.22}
f_n(x):=P_{2n}(x)H_{\be,n}(x)
:=P_{2n}(x)
\left[\frac{\sin(\be\vert x\vert/n)}{\be\vert x\vert/n}\right]^
{2n+2\lceil m/(2q)\rceil +2}.
\eeq
Below, we study certain properties of $f_n$.

\begin{property}\label{P3.8}
For any compact set $K\subset \R^m$,
\beq\label{E3.23}
\lim_{n\to\iy}\left\|f-f_n\right\|_{L_q(K)}=0.
\eeq
\end{property}
\begin{proof}
Since $0\le H_{\be,n}(x)\le 1$ for $x\in\R^m$, we have
\bna\label{E3.24}
\left\|f-f_n\right\|_{L_q(K)}
&=&\left\|\left(f-P_{2n}\right)
+\left(P_{2n}-f\right)\left(1-H_{\be,n}\right)
+f\left(1-H_{\be,n}\right)
\right\|_{L_q(K)}\nonumber\\
&\le& 2\left\|f-P_{2n}\right\|_{L_q(K)}
+\left\|f\right\|_{L_q(K)}
\left\|1-H_{\be,n}\right\|_{L_\iy(K)}.
\ena
Next, $1-H_{\be,n}(x)\le C(\be,m,q)\vert x\vert^2/n,
\,x\in\R^m$, by an elementary inequality
$1-\left(\g^{-1}\sin \g\right)^{2N}\le N\g^2/3,
\,\g\in\R^1,\,N\in\N$.
Thus \eqref{E3.23} immediately follows from \eqref{E3.24}
and inequality \eqref{E3.12} of Lemma \ref{L3.6}.
\end{proof}

\begin{property}\label{P3.9}
The following inequalities are valid:
\bna
&&\left\vert f_n(x)\right\vert
\le Ce^{-2n\vep/\tau}\left(\frac{w^*n}{\tau\vert x\vert}
\right)^{m/q+2},\quad x\in\R^m\setminus(2n/\tau)O_{1/\sa}^m;\label{E3.25}\\
&&\left\|f_n\right\|
_{L_q\left(\R^m\setminus(2n/\tau)O_{1/\sa}^m\right)}
\le Cn^{m/q}e^{-2n\vep/\tau};\label{E3.26}\\
&&\left\|f_n\right\|
_{L_q(\R^m)}<\iy.\label{E3.27}
\ena
Here, constants $C=C(f,\tau,\sa,w^*,m,q)$ are independent of
$n$ and $x$.
\end{property}
\begin{proof}
To prove \eqref{E3.25}, we first use \eqref{E3.12}
of Lemma \ref{L3.6} and Definition \ref{D1.3} to estimate
$P_{2n}$ in \eqref{E3.22}
\bna\label{E3.28}
\left\|P_{2n}\right\|
_{L_\iy\left((2n/\tau)O_{1/\sa}^m\right)}
&\le& \left\|f-P_{2n}\right\|
_{L_\iy\left((2n/\tau)O_{1/\sa}^m\right)}
+\left\|f\right\|
_{L_\iy\left((2n/\tau)O_{1/\sa}^m\right)}\nonumber\\
&\le& C_4e^{-2an}+C_0e^{2n(1+\vep)/\tau}
\le C_5(f,\sa,\vep,\tau,m)e^{2n(1+\vep)/\tau}.
\ena
Next, using Lemma \ref{L2.4} (a) and Lemma \ref{L2.3} (b),
we obtain
from \eqref{E3.28} for $x\in \R^m\setminus(2n/\tau)O_{1/\sa}^m$,
\beq\label{E3.29}
\left\vert P_{2n}(x)\right\vert
\le T_{2n}\left(\frac{\tau\vert x\vert}{w^* n}\right)
\left\|P_{2n}\right\|
_{L_\iy\left((2n/\tau)O_{1/\sa}^m\right)}
\le \left(C_5/2\right)\left(
\frac{2\tau e^{(1+\vep)/\tau}\vert x\vert}
{w^* n}\right)^{2n}.
\eeq
Furthermore, it follows from \eqref{E3.21}, \eqref{E3.22},
and \eqref{E3.29} that
\beq\label{E3.30}
\left\vert f_n(x)\right\vert
\le Ce^{-2n\vep/\tau}\left(\frac{w^* n}{\tau\vert x\vert}
\right)^{2\lceil m/(2q)\rceil+2},
\qquad x\in \R^m\setminus(2n/\tau)O_{1/\sa}^m,
\eeq
where $C=\left(C_5/2\right)\left(2e^{-(1+2\vep)/\tau}\right)^
{2\lceil m/(2q)\rceil+2}$.
Since
\beq\label{E3.31}
\R^m\setminus(2n/\tau)O_{1/\sa}^m \subseteq
\{x\in\R^m:\vert x\vert >w^*n/\tau\},
\eeq
inequality
\eqref{E3.25} is a direct consequence of \eqref{E3.30}
and \eqref{E3.31},
while \eqref{E3.26} immediately follows from
\eqref{E3.25} and \eqref{E3.31}.
Finally, \eqref{E3.27} is an immediate
consequence of \eqref{E3.26}.
\end{proof}

\begin{property}\label{P3.10}
Let $\Om:=\left\{X_\nu\right\}_{\nu=1}^\iy$
 be a $\de_1$-packing net for $\R^m$
 (see Definition \ref{D1.2}).
 In addition, let $n\in\N,\,n\ge \de_1\tau/w^*$, and
 $\Om(n):=\Om
 \cap \left(\R^m\setminus(2n/\tau)O_{1/\sa}^m\right)$.
 Then the following inequalities are valid:
 \bna
 &&\left(\sum_{X_\nu\in \Om(n)}
 \left\vert f_n(X_\nu)\right\vert^q\right)^{1/q}
 \le Cn^{m/q}e^{-2n\vep/\tau},\label{E3.32}\\
 &&\left(\sum_{\nu=1}^\iy
 \left\vert f_n(X_\nu)\right\vert^q\right)^{1/q}
 \le \left(\sum_{\nu=1}^\iy
 \left\vert f(X_\nu)\right\vert^q\right)^{1/q}
 +Cn^{m/q}e^{-bn},\label{E3.33}
 \ena
 where the constants $C=C(f,\sa,\vep,\tau,m,q,\de_1)$
 and $b=b(\tau,\vep)>0$ are independent of $n$.
\end{property}
\begin{proof}
We first prove estimate \eqref{E3.32}.
If $q=\iy$, then \eqref{E3.32} immediately follows from
\eqref{E3.25} of Property \ref{P3.9} and \eqref{E3.31}.
If $q\in[1,\iy)$, then by \eqref{E3.25} and \eqref{E3.31},
\bna\label{E3.34}
\left(\sum_{X_\nu\in \Om(n)}
 \left\vert f_n(X_\nu)\right\vert^q\right)^{1/q}
 &\le& Ce^{-2n\vep/\tau}S_n\nonumber\\
 &:=&Ce^{-2n\vep/\tau}
 \left(\sum_{X_\nu\in \Om(n),\vert X_\nu\vert\ge w^*n/\tau}
 \left(\frac{w^*n}{\tau\vert X_\nu\vert}\right)^{m+2q}\right)
 ^{1/q}.
\ena
To estimate $S_n$, we introduce the finite sets
\ba
\Om_l:=\left\{X_\nu\in \Om:2^lw^*n/\tau
\le\left\vert X_\nu\right\vert
< 2^{l+1}w^*n/\tau\right\},
\qquad l\in Z^1_+.
\ea
Any $X_\nu\in \Om$
with $\vert X_\nu\vert\ge w^*n/\tau$
belongs to $\Om_l$ for a certain $l\in Z^1_+$
and, in addition,
$\frac{w^*n}{\tau\vert X_\nu\vert}
\le 2^{-l}$.
Then it follows from \eqref{E3.34} that
\beq\label{E3.35}
S_n\le \left(\sum_{l=0}^\iy \mathrm{card}(\Om_l)2^{-l(m+2q)}
\right)^{1/q}.
\eeq
It remains to estimate $\mathrm{card}(\Om_l),\,l\in Z^1_+$.
We first recall that $\Om$ is
a $\de_1$-packing net for $\R^m$, i.e.,
the family of open cubes
$\left\{\mathring{Q}_{\de_1/2}^m\left(X_\nu\right)\right\}_{\nu=1}^\iy$
and therefore, the family of open balls
$\left\{\mathring{\BB}^m_{\de_1/2}\left(X_\nu\right)\right\}_{\nu=1}^\iy$
are pairwise disjoint by Lemma \ref{L2.1b} (b).
Setting now $R(l):=2^{l}w^*n/\tau,\,l\in Z^1_+$, we see that $R(l+1)+\de_1/2<R(l+2)$
by the condition $n\ge \de_1\tau/w^*$. Then we obtain for $l\in Z^1_+$,

\bna\label{E3.36}
\mathrm{card}(\Om_l) \left\vert \BB^m_{\de_1/2}\right\vert_m
&=& \sum_{X_\nu\in \Om_l}
\left\vert \BB^m_{\de_1/2}\left(X_\nu\right)\right\vert_m
\le \left\vert \BB^m_{R(l+1)+\de_1/2}\right\vert_m\nonumber\\
&\le& \left\vert \BB^m_{R(l+2)}\right\vert_m
\le C(\tau,m,w^*)2^{l m}n^m.
\ena
Collecting estimates \eqref{E3.34}, \eqref{E3.35}, and
\eqref{E3.36}, we arrive at \eqref{E3.32}.

Next, we prove \eqref{E3.33}.
Using \eqref{E3.22} and Lemma \ref{L3.6}, we have
\bna\label{E3.37}
\left(\sum_{X_\nu\in (2n/\tau)O_{1/\sa}^m}
 \left\vert f_n(X_\nu)\right\vert^q\right)^{1/q}
 &\le& \left(\sum_{X_\nu\in (2n/\tau)O_{1/\sa}^m}
 \left\vert P_{2n}(X_\nu)\right\vert^q\right)^{1/q}
 \nonumber\\
 &\le& \left(\sum_{X_\nu\in (2n/\tau)O_{1/\sa}^m}
 \left\vert f(X_\nu)\right\vert^q\right)^{1/q}
 +C_4\g_n^{1/q}e^{-2an},
\ena
where $\g_n:=\mathrm{card}\left(\Om
\cap(2n/\tau)O_{1/\sa}^m\right)$.
Furthermore, setting
$R:=2(n/\tau)d\left(O_{1/\sa}^m\right)=4n/(\tau \sa)$ by \eqref{E1.0},
similarly to \eqref{E3.36}
we obtain
\beq\label{E3.38}
\g_n
\le \left\vert \BB^m_{\de_1/2}\right\vert_m^{-1}
 \left\vert \BB^m_{R+\de_1/2}\right\vert_m
\le C(\sa,\tau,m,\de_1)n^m.
\eeq
Thus \eqref{E3.33} follows from \eqref{E3.32}, \eqref{E3.37},
and \eqref{E3.38}.
\end{proof}

\section{Proofs of Main Results}\label{S4n}
\setcounter{equation}{0}
\noindent
\emph{Proof of Theorem \ref{T1.1}.}
(a) Without loss of generality, we can assume that $f\in L_q(\R^m)$.
Note first that by Definition \ref{D1.2a}, any cube from
the family of open cubes
$G=\left\{\mathring{Q}_{\de_1/4}^m
\left(X_\nu\right)\right\}_{\nu=1}^\iy$
has the nonempty intersection with no more than $N$
 sets from $G$, not counting the cube itself.
 Then by Lemma \ref{L2.2} (a), there exists a partition
 $\left\{G_j\right\}_{j=1}^{N+1}$ of $G$
 with pairwise disjoint sets in each subfamily
 $G_j=\left\{\mathring{Q}_{\de_1/4}^m
\left(X_{\nu}^{(j)}\right)\right\}_{\nu=1}^\iy,
 \,1\le j\le N+1$.
In addition, there exist points
$Y_\nu^{(j)}\in Q_{\de_1/4}^m\left(X_\nu^{(j)}\right),\,
\nu\in\N$,
such that for $q\in[1,\iy)$ and each $j,\,1\le j\le N+1$,
\bna\label{E4.1}
\|f\|_{L_q(\R^m)}
&\ge& \left(\int_{\bigcup_{\nu=1}^\iy Q_{\de_1/4}^m\left(X_\nu^{(j)}\right)}
\vert f(x)\vert^q\,dx\right)^{1/q}
\nonumber\\
&=&\left(\de_1/2\right)^{m/q}\left(\sum_{\nu=1}^\iy
\left\vert f\left(Y_\nu^{(j)}\right)\right\vert^q\right)^{1/q}
\nonumber\\
&\ge& \left(\de_1/2\right)^{m/q}\left(
\sum_{\nu=1}^\iy\left\vert f\left(X_\nu^{(j)}\right)\right\vert^q\right)^{1/q}
\nonumber\\
&-&\left(\de_1/2\right)^{m/q}
\left\vert\left(\sum_{\nu=1}^\iy\left\vert f\left(X_\nu^{(j)}\right)\right\vert^q\right)^{1/q}
-\left(\sum_{\nu=1}^\iy\left\vert f\left(Y_\nu^{(j)}\right)\right\vert^q\right)^{1/q}
\right\vert.
\ena
Next, it follows from \eqref{E4.1}
and Lemma \ref{L2.8} for $h=\de_1/4$ that
\bna\label{E4.1b}
\left(1+C(m,q)\max\left\{\de_1\sa,(\de_1\sa)^d\right\}\right)^q\|f\|_{L_q(\R^m)}^q
 \nonumber\\
 \ge \left(\de_1/2\right)^{m}
 \sum_{\nu=1}^\iy \left\vert f\left(X_\nu^{(j)}\right)
 \right\vert^q,\qquad 1\le j\le N+1.
\ena
Finally, adding all inequalities \eqref{E4.1b}
for $1\le j\le N+1$, we arrive at
\eqref{E1.3.1} and \eqref{E1.3.2}.
\vspace{.12in}\\
(b) and (c) Since the premises of the statements
do not include the assumption
of $f\in L_q(\R^m),\,q\in[1,\iy]$,
we first prove \eqref{E1.3.4} and \eqref{E1.3.5b}
 for the functions $f_n,\,n\in\N$,
constructed in Section \ref{S3.2n}, and then establish (b) and (c)
by passing to the limit as $n\to\iy$.

We recall that functions
$f_n(\cdot)=f_n(f,m,q,\be,\cdot),\,n\in\N$, are defined by
\eqref{E3.22}, where $\be=\be(\tau,\vep,w^*)$
 is defined by  \eqref{E3.21},
and numbers $\tau\ge 4$ and $\vep>0$ are fixed.
In addition, let us set
$\sa_*:=8e^{1/4}\sqrt{m}\sa$.
Note that
\beq\label{E4.1cc}
\sa<\sa_*<11\sqrt{m}\sa.
\eeq
In this proof we set $\tau=4$, so
$\lim_{\vep\to 0}\be(4,\vep,w^*)=\sa_*/2$.
Then it follows from \eqref{E3.22} that $f_n$ is an entire function
 of spherical type $\sa_n=\sa_*+O(1/n)$ as $n\to\iy$, and, in addition,
 $f_n\in L_q(\R^m),\,n\in\N$, by \eqref{E3.27}.
 Hence
 $f_n\in B_{\sa_n,m}\cap L_q(\R^m),
 \,n\in\N,\,q\in[1,\iy]$.

Let us first discuss  statement (c), i.e., the case of $q=\iy$ and
 $\sup_{\nu\in \N}\left\vert f\left(X_\nu\right)\right\vert<\iy$
 with $\de$, satisfying the condition $11m^{3/2}\de\sa\le 1$.
 Then $m\de\sa_*<1$ by \eqref{E4.1cc}.
 By Definition \ref{D1.1}, for any $x\in\R^m$ there exists
 $X_\nu\in\Omega$ such that
 $\left\|x-X_\nu\right\|_\iy<\de$.
 Therefore, using the mean value theorem and Bernstein-type inequality
 \eqref{E2.18}, we obtain
 for $\nabla:=\left(\partial/\partial x_1,\ldots,
 \partial/\partial x_m\right)$
 \bna\label{E4.2}
 \left\vert f_n\left(x\right)\right\vert
 &\le& \left\vert f_n\left(X_\nu\right)\right\vert
 +\sup_{x\in\R^m}
  \left\|\nabla f_n(x)\right\|_{1}
 \left\|x-X_\nu\right\|_{\iy}\nonumber\\
 &\le& \left\vert f_n\left(X_\nu\right)\right\vert
 +m \de\sa_n\left\| f_n\right\|_{L_\iy(\R^m)}.
 \ena
 Note that $1-m\de\sa_n>0$
 for a large enough
 $n\in\N$ and for a small enough $\vep>0$
 since $1-m\de\sa_*>0$.
 Then for these $n$ and $\vep$,
 the following inequality is a consequence of \eqref{E4.2}:
 \beq\label{E4.3}
 \left\| f_n\right\|_{L_\iy(\R^m)}
 \le \frac{\sup_{\nu\in \N}
 \left\vert f_n\left(X_\nu\right)\right\vert}
 {1-m \de\sa_n}.
 \eeq
 Using relations \eqref{E3.23} and \eqref{E3.33}
 for $q=\iy$ of Properties
 \ref{P3.8} and \ref{P3.10}, respectively,
 we obtain from \eqref{E4.3} that
 for any $x\in\R^m$,
 \ba
 \vert f(x)\vert
 \le \limsup_{n\to\iy}\left\vert f_n(x)\right\vert
 +\lim_{n\to\iy}\left\vert f(x)- f_n(x)\right\vert
 \le \frac{\sup_{\nu\in \N}
 \left\vert f\left(X_\nu\right)\right\vert}
 {1-m \de\sa_*}.
 \ea
 Therefore, $f\in L_\iy(\R^m)$.
 Next, $1-m \sa\de>1-m \sa_*\de>0$ by \eqref{E4.1cc}
  and, replacing $f_n$ with $f$
 and $\sa_n$ with $\sa$ in inequalities \eqref{E4.2} and \eqref{E4.3},
 we arrive at the inequality
 \ba
 \left\| f\right\|_{L_\iy(\R^m)}
 \le \frac{\sup_{\nu\in \N}
 \left\vert f\left(X_\nu\right)\right\vert}
 {1-m \de\sa}.
 \ea
 Thus the proof of statement (c) is completed.

 Next, let us discuss statement (b), i.e., the case of
 $q\in[1,\iy)$ and
 $\left(\sum_{\nu=1}^\iy \left\vert f\left(X_\nu\right)
 \right\vert^q\right)^{1/q}
 <\iy$ with $\de$, satisfying the condition $\sa\de \le C(m,q)$.
 We first apply Lemma \ref{L2.1b} (c) to $\Omega$ and
 find
 a $\de$-packing and a $2\de$-covering  net
 $\Omega^*=\left\{Z_\mu\right\}_{\mu=1}^\iy
 \subseteq\Omega$
  for $\R^m$.
 Note that by Lemma \ref{L2.1b} (a),
 the family of closed cubes
 $G=\left\{Q_{2\de}^m\left(Z_\mu\right)\right\}_{\mu=1}^\iy$
 covers $\R^m$, i.e.,
 $\bigcup_{\mu=1}^\iy Q_{2\de}^m\left(Z_\mu\right)=\R^m$.

 In addition, by Lemma \ref{L2.2} (c),
 there exists a partition $\left\{G_j\right\}_{j=1}^{N+1}$
 of $G$ with pairwise disjoint cubes in each
 $G_j=\left\{Q_{2\de}^m\left(Z_{\mu}^{(j)}\right)\right\}_{\mu=1}^\iy,\,
 1\le j\le N+1$, where
  \beq\label{E4.4}
   N+1\le\left\lfloor 2^m\left(\left({8\de}/{\de}
   \right)^m-1\right)\right\rfloor +1
   \le 2^{4m}.
  \eeq
  Then using again functions $f_n\in B_{\sa_n,m}\cap L_q(\R^m),\,n\in\N$, we have
  \bna\label{E4.5}
  \left\| f_n\right\|_{L_q(\R^m)}
  \le \left(\sum_{j=1}^{N+1}\sum_{\mu=1}^\iy
  \int_{Q_{2\de}^m\left(Z_{\mu}^{(j)}\right)}
   \left\vert f_n(x)\right\vert^q dx\right)^{1/q}
  =(4\de)^{m/q}\left(\sum_{j=1}^{N+1}\sum_{\mu=1}^\iy
  \left\vert f_n\left(Y_{\mu}^{(j)}\right)
  \right\vert^q \right)^{1/q},
  \ena
where $Y_{\mu}^{(j)}\in Q_{2\de}^m\left(Z_{\mu}^{(j)}\right),\,
\mu\in\N,\,1\le j\le N+1$.
Furthermore, applying Lemma \ref{L2.8} for $h=2\de$
and each $j,\,1\le j\le N+1$,
 we obtain
\ba
\left(\sum_{\mu=1}^\iy
  \left\vert f_n\left(Y_{\mu}^{(j)}\right)
  \right\vert^q \right)^{1/q}
  &\le& \left(\sum_{\mu=1}^\iy\left\vert
  f_n\left(Z_{\mu}^{(j)}\right)
  \right\vert^q \right)^{1/q}\\
  &+& C(m,q)\de^{-m/q}
  \max\left\{\de\sa_n,(\de\sa_n)^d\right\}
  \|f_n\|_{L_q(\R^m)}.
\ea
Hence
\bna\label{E4.6}
\sum_{j=1}^{N+1}\sum_{\mu=1}^\iy
  \left\vert f_n\left(Y_{\mu}^{(j)}\right)
  \right\vert^q
  &\le& 2^{q-1}\sum_{\mu=1}^\iy\left\vert
  f_n\left(Z_{\mu}\right)
  \right\vert^q \nonumber\\
  &+& C(m,q)(N+1)\de^{-m}
  \max\left\{(\de\sa_n)^q,(\de\sa_n)^{dq}\right\}
  \|f_n\|_{L_q(\R^m)}^q.
  \ena
  It follows from \eqref{E4.1cc}, \eqref{E4.4}, and the condition $\sa\de\le C(m,q)$
  that there exists a constant $C(m,q)$ such that
  \beq\label{E4.6a}
  1-C(m,q)(N+1)
 \max\left\{(\de\sa_*)^q,(\de\sa_*)^{dq}\right\}>0.
  \eeq
Combining \eqref{E4.5} with \eqref{E4.6} and
\eqref{E4.4}, we have
\bna\label{E4.7}
 \left\| f_n\right\|_{L_q(\R^m)}
 &\le& \frac{(4\de)^{m/q}2^{1-1/q}
 \left(\sum_{\mu=1}^\iy\left\vert
  f_n\left(Z_{\mu}\right)
  \right\vert^q \right)^{1/q}}
  {\left(1-C(m,q)(N+1)
  \max\left\{(\de\sa_n)^q,(\de\sa_n)^{dq}\right\}\right)^{1/q}}\nonumber\\
  &\le& \frac{(4\de)^{m/q}2^{1-1/q}
 \left(\sum_{\nu=1}^\iy\left\vert
  f_n\left(X_{\nu}\right)
  \right\vert^q \right)^{1/q}}
  {\left(1-C(m,q)
 \max\left\{(\de\sa_n)^q,(\de\sa_n)^{dq}\right\}\right)^{1/q}},
\ena
where the denominators in
\eqref{E4.7} are positive by \eqref{E4.6a} for
a large enough
 $n\in\N$ and a small enough $\vep>0$.

 Using now relations \eqref{E3.23} and \eqref{E3.33}
 of Properties
 \ref{P3.8} and \ref{P3.10}, respectively, we obtain
 from \eqref{E4.7} that
 for any compact $K\subset\R^m$,
 \bna\label{E4.8}
 \left\| f\right\|_{L_q(K)}
 &=& \limsup_{n\to\iy}\left\| f_n\right\|_{L_q(K)}
 +\lim_{n\to\iy}\left\| f-f_n\right\|_{L_q(K)}
 \nonumber\\
 &\le& \frac{(4\de)^{m/q}2^{1-1/q}
 \left(\sum_{\nu=1}^\iy\left\vert
  f\left(X_{\nu}\right)
  \right\vert^q \right)^{1/q}}
  {\left(1-C(m,q)
 \max\left\{(\de\sa_*)^q,(\de\sa_*)^{dq}\right\}\right)^{1/q}}.
 \ena
 Since the right-hand side of \eqref{E4.8}
 is independent of $K$, we  conclude that $f\in L_q(\R^m)$.
 Next, replacing $f_n$ with $f$
 and $\sa_n$ with $\sa$ in inequalities
 \eqref{E4.5}, \eqref{E4.6}, and \eqref{E4.7},
 we arrive at \eqref{E1.3.4} with estimate \eqref{E1.3.5}
 for $C_2$.
 Note that $C_2>0$ by estimates \eqref{E4.6a}, \eqref{E4.4}, and \eqref{E4.1cc}.
 Thus the proof of statement (b) is completed.
 \hfill $\Box$
 \vspace{.12in}\\
 \emph{Proof of Corollary \ref{C1.2a}.}
 The corollary follows from Corollary \ref{C1.2} and Lemma \ref{L2.1b} (c).
 \hfill $\Box$
 \vspace{.12in}\\
 \emph{Proof of Theorem \ref{T1.1a}.}
 (a) We first note that the function
 $f_0\in B_{\sa,m}\cap L_q(\R^m), q\in[1,\iy)$,
satisfies the relation
$\lim_{\vert x\vert\to\iy} f_0(x)=0$
(see, e.g., \cite[Theorem 3.2.5]{N1969}).
Therefore, $f_0\in L_\iy(\R^m),\,
f_0$ is not identically zero,
and there exists $x_0\in\R^m$ such that
$\left\|f_0\right\|_{L_\iy(\R^m)}
=\left\vert f_0\left(x_0\right)\right\vert$.

Next, for $\de_1^*=\de_1^*:=1/(m\sa)$, the following
inequality is valid:
\beq\label{E4.1c}
\inf_{x\in Q_{\de_1^*/2}^m(x_0)}
\left\vert f_0\left(x\right)\right\vert
\ge (1/2)\left\vert f_0\left(x_0\right)\right\vert,
\eeq
since by the mean value theorem and a
Bernstein-type inequality
\eqref{E2.18},
\ba
\left\vert f_0\left(x_0\right)\right\vert
\le \left\vert f_0\left(x\right)\right\vert
+m\sa\left\vert f_0\left(x_0\right)\right\vert
\left\|x-x_0\right\|_{\iy}
\ea
 (cf. \eqref{E4.2}).
Finally, setting $f_\nu(x)
:=f_0\left(x+x_0-X_\nu\right)$,
 we obtain
by \eqref{E1.3.1} and \eqref{E4.1c},
\ba
\left\|f_0\right\|_{L_q(\R^m)}
=\left\|f_\nu\right\|_{L_q(\R^m)}
&\ge& C_1 \left(\sum_{\mu=1}^\iy
\left\vert f_\nu\left(X_\mu\right)
 \right\vert^q\right)^{1/q}
 \ge C_1 \left(\sum_{X_\mu\in
 \mathring{Q}_{\de_1^*/2}^m(X_\nu)}
\left\vert f_\nu\left(X_\mu\right)
 \right\vert^q\right)^{1/q}\\
 &\ge& (1/2)C_1
 \left\|f_0\right\|_{L_\iy(\R^m)}
 \left(\mathrm{card}\left(\Omega\cap
 \mathring{Q}_{\de_1^*/2}^m
 \left(X_\nu\right)\right)\right)^{1/q}
 ,\qquad \nu\in\N.
\ea
Thus statement (a) is established
with $\de_1\in (0,\de_1^*]$ and $N$, satisfying inequalities \eqref{E1.3.5c}.
\vspace{.12in}\\
(b) Let $\Omega^*:=\left\{X_\nu^*\right\}_{\nu=1}^\iy$
 be a $(\de_1,N)$-packing net
 for $\R^m,\,N\in\Z^1_+$.
 Note that by Definition \ref{D1.2a}, any cube from
the family of open cubes
$G=\left\{\mathring{Q}_{\de_1/2}^m
\left(X_\nu^*\right)\right\}_{\nu=1}^\iy$
has the nonempty intersection with no more than $N$
 sets from $G$, not counting the cube itself.
 Then by Lemma \ref{L2.2} (a), there exists a partition
 $\left\{G_j\right\}_{j=1}^{N+1}$ of $G$
 with pairwise disjoint sets in each subfamily
 $G_j=\left\{\mathring{Q}_{\de_1/2}^m
\left(X_{\nu}^{*(j)}\right)\right\}_{\nu=1}^\iy,
 \,1\le j\le N+1$.
 Next setting
 \ba
 \Om_l^{(j)}\left(\de_1\right)
 :=\left\{X_\nu^{*(j)}
 : 2^l\de_1\le \left\vert X_\nu^{*(j)}
 \right\vert < 2^{l+1}\de_1\right\},
 \qquad l\in\Z^1_+,\quad 1\le j\le N+1,
 \ea
 similarly to \eqref{E3.36}
we obtain
\beq\label{E4.1.2c}
\mathrm{card}\left(\left(\Om_l^{(j)}\right)
\left(\de_1\right)\right)
\le \left\vert \BB^m_{\de_1/2}\right\vert_m^{-1}
\left\vert \BB^m_{2^{l+2}\de_1}\right\vert_m
=2^{m(l+2.5)}.
\eeq
 In addition, we introduce the following entire function
 \beq\label{E4.1.2d}
 f_0(x):=\left(\frac{\sin(\sa\vert x\vert/\g)}
{\vert x\vert/\g}\right)^\g
 \eeq
 of spherical type $\sa$ that satisfies the relations
 $f_0\in B_{\sa,m}$ and
 $\left\|f_0\right\|_{L_q(\R^m)}=C_6(m,q)\sa^{\g-m/q}$.
 Here,  $\g$ is defined in \eqref{E1.3.5e}.
 Then for $n\in\Z^1_+$ we obtain from \eqref{E4.1.2d} and \eqref{E4.1.2c}
 \bna\label{E4.1.3c}
 &&\sum_{\left\vert X_\nu^{*}\right\vert\ge 2^n\de_1}
 \left\vert f_0\left(X_\nu^{*}\right)\right\vert^q
 =\sum_{j=1}^{N+1}
 \sum_{\left\vert X_\nu^{*(j)}\right\vert\ge 2^n\de_1}
 \left\vert f_0\left(X_\nu^{*(j)}\right)\right\vert^q
 \nonumber\\
 &&= \sum_{j=1}^{N+1} \sum_{l=n}^\iy
 \sum_{X_\nu^{*(j)}\in \Om_l^{(j)}\left(\de_1\right)}
 \left\vert f_0\left(X_\nu^{*(j)}\right)\right\vert^q
 \nonumber\\
 &&\le C_7(m,q)\de_1^{-\g q}\sum_{j=1}^{N+1}\sum_{l=n}^\iy
 \mathrm{card}\left(\Om_l^{(j)}\left(\de_1\right)\right)
 2^{-\g ql}
 \nonumber\\
 &&\le C_{8}(m,q)\de_1^{-\g q}(N+1)2^{-(\g q-m)n}.
 \ena
 Furthermore, recall that $\de^*$ is defined in
  \eqref{E1.3.5e},
 and let us define the constant $C$ in \eqref{E1.3.5e} by
 $C(m,q):=2\left(C_8 C_6^{-q}\right)^{1/(\g q-m)}$.
 To prove statement (b), it suffices to show that
 if $\de>\de^*$,
 then $\Omega\cap \mathring{Q}_\de^m(y)\ne\emptyset$
 for any $y\in\R^m$.

 Indeed, assume that $\de>\de^*$ and
 there exists $y\in\R^m$ such that
 $\Omega\cap \mathring{Q}_\de^m(y)=\emptyset$.
 Let us set $X_\nu^*:=X_\nu-y,\,\nu\in\N$. Then
 $\Omega^*:=\left\{X_\nu^*\right\}_{\nu=1}^\iy$
 is a $(\de_1,N)$-packing net
 for $\R^m,\,N\in\Z^1_+$.
 Next, note that $\de^*\ge \de_1$ by \eqref{E1.3.5e},
 so there exists $n\in\Z^1_+$ such that
 $\de\in\left[2^n\de_1,2^{n+1}\de_1\right)$.
 Setting now $f_y(x):=f_0(x-y)$,
 we see from
 \eqref{E1.3.4} and \eqref{E4.1.3c} that
 \ba\label{E4.1.4c}
 &&C_{8}\de_1^{-\g q}(N+1)(2\de_1/\de)^{\g q-m}
 \ge C_{8}\de_1^{-\g q}(N+1)2^{-(\g q-m)n}
 \ge \sum_{\left\vert X_\nu^{*}\right\vert\ge 2^n\de_1}
 \left\vert f_0\left(X_\nu^{*}\right)\right\vert^q\\
 &&\ge \sum_{\left\vert X_\nu^{*}\right\vert\ge \de}
 \left\vert f_0\left(X_\nu^{*}\right)\right\vert^q
 =\sum_{\nu=1}^\iy
 \left\vert f_y\left(X_\nu\right)\right\vert^q
 \ge C_2^q\left\|f_y\right\|_{L_q(\R^m)}^q
 =\left(C_2C_6\right)^{q}\sa^{\g q-m}.
 \ea
 Hence $\de\le \de^*$.
 This contradiction shows  that for
 any $\de\in(\de^*,\iy),\,
 \Omega$ is a $\de$-covering net
 for $\R^m$, by
 Definition \ref{D1.1}.
 \vspace{.12in}\\
 (c)  We first note that for any $\de>0$, the entire function
$
f_0(x):=\sin(\sa\vert x\vert)
/\vert x\vert
$
of spherical type $\sa$ satisfies the following
 relations:
 \beq\label{E4.1.1c}
 f_0\in B_{\sa,m},\qquad
 \left\|f_0\right\|_{L_\iy(\R^m)}=\sa,\qquad
 \left\|f_0\right\|_
 {L_\iy\left(\R^m\setminus \mathring{Q}^m_\de
 \right)} \le 1/\de.
 \eeq
 Next, if $\de>\left(C_3\sa\right)^{-1}$,
 where $C_3$ is the constant from \eqref{E1.3.5b},
 then $\Omega\cap \mathring{Q}_\de^m(y)\ne\emptyset$
 for any $y\in\R^m$. Indeed, assume that
 there exists $y\in\R^m$ such that
 $\Omega\cap \mathring{Q}_\de^m(y)=\emptyset$.
 Then setting $f_y(x):=f_0(x-y)$, we see from
 \eqref{E1.3.5b} and \eqref{E4.1.1c} that
 \ba
 1/\de
 \ge  \left\|f_y\right\|_
 {L_\iy\left(\R^m\setminus \mathring{Q}_
 {\de}^m(y)\right)}
 \ge \sup_{\nu\in\N}\left\vert
 f_y\left(X_\nu\right)\right\vert
 \ge C_3\left\|f_y\right\|_
 {L_\iy\left(\R^m\right)}=C_3\sa.
 \ea
 This contradiction shows  that for
 any $\de\in \left(\left(C_3\sa\right)^{-1},\iy\right),\,
 \Omega$ is a $\de$-covering net
 for $\R^m$, by
 Definition \ref{D1.1}.
 \vspace{.12in}\\
 \emph{Proof of Theorem \ref{T1.3}.}
 Using first Lemma \ref{L3.7} for
 $A=D_n\|f\|_{L_\iy(Q^m_b)},\,n=n(N)\in \ES$,
 we see that there exists a polynomial
 $P_n\in\QQ_{n,m}\subseteq \PP_{mn,m}$
 such that
 \beq\label{E4.9}
 \left\|f-P_n\right\|_{L_\iy(Q^m_b)}
 \le C D_n\|f\|_{L_\iy(Q^m_b)}e^{n\psi(\tau)},
 \eeq
 where $C=C(m)$.
 Next, according to Theorem \ref{T1.0},
 for any $\g>0$ there exist
  a constant
$C=C(b,\tau,m,\g)$
and a finite set  $\left\{X_1,\ldots, X_\La\right\}
\subset Q^m_b$
with $\La\le Cn^m$ such that
\bna\label{E4.10}
\|f\|_{L_\iy(Q^m_b)}
&\le& \left\|P_n\right\|_{L_\iy(Q^m_b)}
+\left\|f-P_n\right\|_{L_\iy(Q^m_b)}\nonumber\\
&\le& \sqrt{1+\g}\max_{1\le j\le \La}
\left\vert P_n\left(X_j\right)\right\vert
+\left\|f-P_n\right\|_{L_\iy(Q^m_b)}\nonumber\\
&\le& \sqrt{1+\g}\max_{1\le j\le \La}
\left\vert f\left(X_j\right)\right\vert
+\left(1+\sqrt{1+\g}\right)
\left\|f-P_n\right\|_{L_\iy(Q^m_b)}.
\ena
Finally, choosing by \eqref{E1.3.7} an integer
$N_0=N_0(b,\tau,m,\g)\in\N$ such that
\ba
\left(1-\left(1+\sqrt{1+\g}\right)
C D_{n(N)}e^{n(N)\psi(\tau)}\right)^{-1}<\sqrt{1+\g}
\ea
for $N\ge N_0$,
we arrive at \eqref{E1.3.8} for
$n\ge n_0:=n(N_0)$ from
\eqref{E4.9} and \eqref{E4.10}.
\hfill $\Box$
\vspace{.12in}\\
 \emph{Proof of Theorem \ref{T1.5}.}
 Setting first $U(y):=\sum_{k\in\Z^m_+
 \cap Q^m_N}c_ky^k\in\QQ_{N,m}$, where
 $c_k,\,k\in\Z^m_+\cap Q^m_N$, are the coefficients
 of the exponential polynomial $E_N$ defined in
 \eqref{E1.3.10},
 and applying Lemma \ref{L2.6} (b) to $U$ for $n=N$,
 we obtain
 \ba
 \sum_{k\in\Z^m_+\cap Q^m_N}\left\vert c_k\right\vert
 &\le&\left(\frac{e^{b/4}+e^{-b/4}}{e^{b/4}-e^{-b/4}}\right)^{mN}
 \|U\|_{L_\iy\left(\left[e^{-b},e^b\right]^m\right)}\nonumber\\
 &=&\left(\frac{e^{b/4}+e^{-b/4}}{e^{b/4}-e^{-b/4}}\right)^{mN}
 \left\|E_N\right\|_{L_\iy(Q^m_b)}.
 \ea
 Hence
 \beq\label{E4.11}
\left\vert E_N(w)\right\vert
\le \left(\frac{e^{b/4}+e^{-b/4}}{e^{b/4}-e^{-b/4}}\right)^{mN}
 \left\|E_N\right\|_{L_\iy(Q^m_b)}
 \exp\left(N\sum_{j=1}^m\left\vert w_j\right\vert\right),
 \qquad w\in\CC^m.
 \eeq
Next, given $\tau>\g_0=1.5088\ldots$, let us set
$n(N):=\lceil Nmb\tau\rceil,\,N\in\N$.
We also recall that $\psi(\tau)$ is defined
by \eqref{E1.3.5a} and $\psi(\tau)<0$ for $\tau>\g_0$.
Then it follows from \eqref{E4.11} that
\beq\label{E4.12}
\left\vert E_N(w)\right\vert
\le D_n
 \left\|E_N\right\|_{L_\iy(Q^m_b)}
 \exp\left(\frac{n}{mb\tau}
 \sum_{j=1}^m\left\vert w_j\right\vert\right),
 \qquad w\in\CC^m.
 \eeq
 with
 \ba
 D_n:=
 \left(\frac{e^{b/4}+e^{-b/4}}{e^{b/4}-e^{-b/4}}\right)^{mN}.
 \ea
 Then
 \beq\label{E4.13}
D_ne^{n\psi(\tau)}
\le \left(\frac{e^{b/4}+e^{-b/4}}{e^{b/4}-e^{-b/4}}\right)^{mN}
e^{Nmb\tau\psi(\tau)}
:=e^{mNG(\tau,b)},
\eeq
where
\ba
G(\tau,b)
:=b\left(\sqrt{1+\tau^2}
-\tau\log\left(\tau+\sqrt{1+\tau^2}\right)\right)
+\log \frac{e^{b/4}+e^{-b/4}}{e^{b/4}-e^{-b/4}}.
\ea

Since $G(\cdot,b)$ is a strictly decreasing function
on $(0,\iy)$ for a fixed $b$ and
$\lim_{\tau\to\iy}G(\tau,b)=-\iy$,
there exists the unique solution
$\tau_0=\tau_0(b)\in(\g_0,\iy)$ to
the equation $G(\tau,b)=0$.
Then by \eqref{E4.13} for $\tau>\tau_0$,
\beq\label{E4.14}
\lim_{N\to\iy}D_{n(N)}e^{n(N)\psi(\tau)}
=\lim_{N\to\iy}e^{mNG(\tau,b)}=0.
\eeq
Relations \eqref{E4.12} and \eqref{E4.14}
show that conditions \eqref{E1.3.6} and
\eqref{E1.3.7} of Theorem \ref{T1.3}
hold for $f=E_N$ and $\tau>\tau_0$.
Thus Theorem \ref{T1.5} follows from
Theorem \ref{T1.3}.
\hfill $\Box$
\vspace{.12in}\\
\textbf{Acknowledgements.}
We are grateful to Andr\'{a}s Kro\'{o}
for the provision of references
\cite{K2019} and \cite{DP2023}.

\end{document}